\documentclass[11pt,reqno]{amsart}
\usepackage{fullpage}
\usepackage{dsfont}
\usepackage{amsmath}
\usepackage{amsfonts}
\usepackage{amssymb}
\usepackage{graphicx}
\usepackage{mathrsfs}
\usepackage{epsfig}
\usepackage{amsthm}
\usepackage{verbatim}
\usepackage{graphicx}
\usepackage{color}
\usepackage{url}

\usepackage[colorlinks,citecolor=black,linkcolor=black,
            bookmarksopen,
            bookmarksnumbered
           ]{hyperref}

\newcommand{\black}{\color{black}}

\newcommand{\beq}{\begin{equation}}
\newcommand{\eeq}{\end{equation}}
\newcommand{\e}{\mathrm{e}}
\newcommand{\dd}{\,\mathrm{d}}

\newcommand{\taul}{\tau}

\newcommand{\uta}{u_{\tau}}

\newtheorem{theorem}{Theorem}[section]

\newtheorem{lemma}[theorem]{Lemma}
\newtheorem{cor}[theorem]{Corollary}
\newtheorem{proposition}[theorem]{Proposition}
\newtheorem{rem}[theorem]{Remark}

\allowdisplaybreaks[3]

\author{Fr\'ed\'eric Rousset}
\address{Universit\'e Paris-Saclay, CNRS,   Laboratoire de Math\'ematiques d'Orsay (UMR 8628),  91405 Orsay Cedex, France (F. Rousset)}
\email{frederic.rousset@universite-paris-saclay.fr}

\author{Katharina Schratz}
\address{LJLL (UMR 7598), Sorbonne Universit\'e, UPMC, 4 place Jussieu, 75005, Paris, France (K. Schratz)}
\email{katharina.schratz@ljll.math.upmc.fr}
 
\begin{document}
\begin{abstract}
We consider various filtered time discretizations of the  periodic Korteweg--de Vries equation: a filtered exponential integrator, a filtered Lie splitting scheme as well as a filtered resonance based  discretisation and establish  error estimates at low regularity. 
  Our analysis is based on discrete Bourgain spaces and allows  to prove  convergence  in $L^2$ for rough data $u_{0} \in H^s,$ $s>0$ with an explicit
 convergence rate.
\end{abstract}

\title[]{Convergence error estimates at  low regularity \\ for time discretizations of KdV}

\maketitle

\section{Introduction}

We consider the Korteweg--de Vries (KdV) equation
\begin{equation}\label{KdV}
\begin{aligned}
\partial_t u(t,x) + \partial_x^3 u(t,x) = -\frac12 \partial_x u^2(t,x), \quad (t,x) \in \mathbb{R} \times \mathbb{T}
\end{aligned}
\end{equation}
with initial data $u(0,x) = u_0(x)$. In the last decades a large variety of numerical schemes was proposed to approximate the time dynamics of KdV solutions; see, e.g.,  \cite{Clem,HLR12,HKRT12,HKR99,HoS16,Klein06,OSu,T74,WuZ,WuZ2}. Their error analysis is so far restricted to smooth Sobolev spaces  and 
  requires smooth solutions $u$ at least in $H^s$ with $s > 3/2$.
 For a long time it was therefore an open question whether  convergence, even  with arbitrarily small rate, can be achieved  for rough data 
 \begin{equation}\label{rough}
u_{0} \in H^s \qquad 0< s\le 3/2.
\end{equation} 
Our aim  is to address this question for a general class of schemes including the Lie splitting and exponential integrator methods.

More precisely, we consider  the filtered exponential integrator
\begin{equation}\label{expint}
u^{n+1}= e^{ - \tau \partial_x^3} \left[ u^n-  \frac\tau2 \varphi_1(\tau \partial_x^3) 
\Pi_{\taul} \partial_x \left(  \Pi_{\taul} u^n\right) ^2\right], \qquad \varphi_1(\delta) = \frac{e^\delta -1}{\delta},
\end{equation}
the filtered Lie splitting  (or Lawson method)
\begin{equation}\label{split}
u^{n+1}= e^{ - \tau \partial_x^3} \left[ u^n-  \frac\tau2  
\Pi_{\taul} \partial_x \left(  \Pi_{\taul} u^n\right) ^2\right],
\end{equation}
as well as the filtered version of the resonance based  scheme introduced in \cite{HoS16}, 
\begin{equation}\label{res}
\begin{aligned}
u^{n+1} & = \mathrm{e}^{-\tau \partial_x^3}  u^n - \frac{1}{6}  \Pi_{\taul}  \Big( \mathrm{e}^{- \tau \partial_x^3} \partial_x^{-1}     \Pi_{\taul}  u^n\Big)^2 + \frac{1}{6} \Pi_{\taul}  \mathrm{e}^{- \tau \partial_x^3}\Big( \partial_x^{-1}  \Pi_{\taul}  u^n\Big)^2,
\end{aligned}
\end{equation}
where the projection operator $ \Pi_{\taul}$  is defined by the Fourier multiplier
\beq
\label{PiKdef}
 \Pi_{\taul}=  \chi(  -i \partial_x \, \tau^{1 \over 3} ),
\eeq
with $\chi= \mathrm{1}_{[-1, 1]}$.  

We initialize with $u^{0}= \Pi_{\tau}u_{0}$.   In the above schemes, $\tau$ denotes the time step size
and $u^n$ aims to approximate the solution $u$ of \eqref{KdV}  at  time $t_{n}= n \tau$.

The unfiltered Strang splitting scheme for KdV was analysed in \cite{HLR12,HKRT12}; under the assumption that the nonlinear part, i.e., Burgers equation $ \partial_t u =-\frac12 \partial_x u^2$,  is solved exactly,  second-order convergence rate for $H^{r+5}$ solutions could be  established in $H^r$ for any $r\geq 1$ (the same analysis would give first-order convergence for $H^{r+3}$ solutions for the  Lie splitting). With the aid of a Rusanov scheme, which roughly corresponds to the introduction of a  diffusion effect and allows to handle the derivative in Burger's nonlinearity,  error estimates for $H^3$ solutions could be furthermore obtained in  \cite{OSu}.  In \cite{Clem}, where a finite difference scheme is studied for   the equation  on the real line $\mathbb{R}$, a convergence result is obtained for data in $H^s$, $s\geq 3/4$. Thereby,  convergence of order $1/42$ holds under the CFL condition $\Delta t  \leq \Delta x^3$ in case of $s=3/4$. The latter  convergence analysis is, however, restricted to the real line  as it heavily relies on a smoothing effect on $\mathbb{R}$ which does not hold on the torus $\mathbb{T}$. The unfiltered resonance based discretisation, that is \eqref{res} with $\Pi_\taul = 1$, was originally introduced in \cite{HoS16}
to allow improved convergence rates for rougher data than  classical schemes allow. More precisely, first-order convergence in $H^1$ for solutions in $H^3$ can be established  for this scheme (\cite{HoS16}). Another   unfiltered resonance based  discretisation of  embedded type was recently  introduced in \cite{WuZ2} which allows first-order convergence in $H^{1/2+\epsilon}$ for solutions in $H^{3/2+\epsilon}$ for any $\epsilon>0$.
 The   convergence analysis in \cite{HoS16,WuZ2},   based on energy type  estimates  and standard product rules in Sobolev spaces,   would not allow us to handle rough data of type~\eqref{rough}
 on the torus (even at the price of a reduced convergence rate) since at least  Lipschitz solutions are needed for the argument.
 The situation is even worse for the unfiltered exponential integrator \eqref{expint} or the   Lie  splitting  \eqref{split}
 without Friedrichs or Rusanov corrections for Burgers (as used in 
 \cite{Clem,OSu}) since the energy method is unconclusive and the schemes seem unstable.

The aim of this paper is to handle in a unified way the three filtered schemes \eqref{expint}, \eqref{split}, \eqref{res}
and to provide convergence estimates  which  allow  us to deal with rough data \eqref{rough}.   In context of nonlinear Schr\"odinger equations  low regularity estimates  could be recently established  with the aid of  discrete Strichartz type estimates (on $\mathbb{R}^d$) and Bourgain type estimates (on $\mathbb{T}$); see \cite{IZ09,Ignat11,ORSBourg,ORS19,ORS20}. In context of the  KdV equation \eqref{KdV} our analysis will also rely on the discrete Bourgain spaces introduced in \cite{ORSBourg}. Nevertheless, 
as in the analysis of the continuous PDE, in order to overcome the loss of derivative in Burger's nonlinearity some new substantial
developments are needed. The presence of the filter $\Pi_{\tau}$ will be crucial to avoid a loss of derivative and to reproduce
at the discrete level the favorable frequency interactions of the KdV equation.

To deal with all three schemes \eqref{expint}-\eqref{res} at the same time we introduce 
\begin{align}\label{Psi}
\Psi_{\tau} (v)= 
- \frac12 
 \int_0^\tau 
\psi_1(s,\partial_x)
 \partial_x \left(   \psi_2(s,\partial_x)  v \right) ^2
ds,
\end{align}
where $\psi_1(s,\partial_x)$ and $ \psi_2(s,\partial_x)$ are Fourier multipliers with bounded symbols  $\psi_{1,2}(s,\partial_x) \in \{ 1, e^{\pm s \partial_x^3}\}$. This notation allows us to express the schemes \eqref{expint}, \eqref{split} and \eqref{res}  in the compact way 
\begin{equation}\label{scheme}
u^{n+1}= e^{ - \tau \partial_x^3}\left[ u^n + \Pi_{\taul} \Psi_{\tau} (\Pi_{\taul} u^n)\right],
\end{equation}
where the choice
\begin{equation*}\label{filtexpInt1}
\psi_1(s,\partial_x) =  e^{ s \partial_x^3}, \quad \psi_2(s,\partial_x) =1
\end{equation*}
corresponds to the exponential integrator \eqref{expint}, while setting
\begin{equation*}\label{filtexpInt2}
\psi_1(s,\partial_x) =  1, \quad \psi_2(s,\partial_x) =1
\end{equation*}
yields the Lie splitting  \eqref{split} and 
\begin{equation}
\label{filtexpInt3}
\psi_1(s,\partial_x) =  e^{ s \partial_x^3},\quad  \psi_2(s,\partial_x) =\mathrm{e}^{-s  \partial_x^3}
\end{equation}
leads to  the resonance based scheme \eqref{res}.

The filtered scheme \eqref{scheme} with the corresponding choice of filter function can be seen as a  classical exponential integrator/Lie splitting/resonance based discretisation applied to the projected KdV equation 
\begin{equation}
\begin{aligned}\label{KdVP}
\partial_t \uta + \partial_x^3  \uta=   -\frac12 \Pi_{\taul} \partial_x \left(  \Pi_{\taul} \uta\right) ^2, \qquad \uta(0) & =  \Pi_{\taul}  u_{0}.
\end{aligned}
\end{equation}

Our main convergence result is the following:

\begin{theorem}\label{maintheo}
For every $T >0$ and $u_{0} \in H^{s_{0}}$,  $s_{0}>0$,  $\int_{\mathbb{T}}u_{0}= 0$, let $u \in \mathcal{C}([0, T], H^{s_{0}})\cap X^{s_0}(T)
$  (we shall define this space in Section \ref{sec:errPro}) be the exact solution of \eqref{KdV} with initial datum $u_{0}$ and denote by $u^n$ the sequence defined by the scheme~\eqref{scheme}.
Then, we have  the following error estimate:
there exist $\tau_{0}>0$ and $C_{T}>0$ such that for every step size $\tau \in (0, \tau_{0}]$
\beq
\label{final1} \| u^n- u(t_{n})\|_{L^2} \leq C_{T} \max(\tau^{s_{0}\over 3}, \tau), \quad 0 \leq n\tau \leq T.
\eeq
\end{theorem}
Note that, we are able to establish a convergence result in $L^2$ with explicit convergence rate for any initial data
 in $H^{s_{0}}$, $s_{0}>0$. For $s_{0} \geq 3$, we recover a classical first order convergence result. Note that even in case of smooth solutions $ s_{0}>3$, the convergence analysis of the schemes~\eqref{expint} and ~\eqref{split}  will require the use of discrete Bourgain
spaces. This is due to the fact that the bilinear estimates in these spaces  are crucial to overcome the derivative in the right hand side in the stability analysis.
 For the resonance based scheme \eqref{res} (and its unfiltered counterpart, i.e., $\Pi_{\taul} =1$), on the other hand, a more standard convergence analysis by energy estimates can be carried out for smooth solutions $s_{0} >3$ without  employing discrete Bourgain spaces (see, e.g., \cite{HoS16}). 
 
 From a similar analysis, we can also obtain  that if $u_{0} \in H^{s_{0}+r}$, $r \geq 0$, $s_{0}>0$,  then we have 
 the error estimate
 $$ \| u^n- u(t_{n})\|_{H^r} \leq C_{T} \max(\tau^{s_{0}\over 3}, \tau).$$
 In the case $u_{0} \in L^2$, we can  deduce from our analysis and  an approximation argument 
 like in \cite{Clem}
 that we have convergence
  without rate for small times, namely that  there exists $T>0$ such that
 $$ \lim_{\tau \rightarrow 0} \sup_{ 0 \leq n \leq T/\tau} \|u(t_{n})- u^n \|_{L^2} = 0.$$


\black

\bigskip

\paragraph*{\bf Outline of the paper.}
The idea is to first analyse the difference between the original KdV equation~\eqref{KdV} and its projected counterpart \eqref{KdVP} on the continuous level; see Section \ref{sec:errPro}. This will then allow us to analyse  the time discretisation error introduced by the  discretisation~\eqref{scheme} applied to the projected equation~\eqref{KdVP}; see Section \ref{sec:err} and \ref{sectionproofmain}.
In Section
 \ref{sectionBourgain}, we introduce the  appropriate discrete Bourgain spaces for the KdV equation and
 establish  their main properties. The proof of the crucial  bilinear estimate stated in Lemma \ref{bourgaindurd}
 is postponed to Section \ref{sectiontechnical}.

\paragraph*{\bf Notations.}
For two expressions $a$ and $b$, we write $a \lesssim b$ whenever $a \leq C b$ holds with some constant $C>0$, uniformly in $\tau \in (0, 1]$.  We further write $a\sim b$ if $b\lesssim a \lesssim b$. When we want to emphasize that $C$ depends on an additional parameter $\gamma$, we write $a \lesssim_{\gamma} b$.
Further, we denote $\langle \,\cdot\, \rangle = ( 1 + | \cdot |^2)^{1 \over 2}$.\black

\section{Error between the solutions of the  exact and projected equation}\label{sec:errPro}

In this section we establish an estimate on the difference between the solutions of the  original KdV equation  \eqref{KdV} and its projected counterpart \eqref{KdVP}. This will yield a bound on
 $$
\Vert u (t) - \uta(t)\Vert_{L^2}.
$$

We shall first recall the main tools that are used to prove local well-posedness
at low regularity for KdV on the torus \cite{Bour93b, KPV, CKSTT} since they are needed to estimate
$ \Vert u (t) - \uta(t)\Vert_{L^2}$.

Let us recall the definition of Bourgain spaces in the setting of the KdV equation. A tempered distribution $u(t,x)$ on $\mathbb{R} \times \mathbb{T}$ belongs to the Bourgain space $X^{s, b}$ if its following norm is finite
\begin{equation*}
\|u\|_{X^{s, b}}=\left(\int_{\mathbb{R}}\sum_{k \in \mathbb{Z}}\left(1+ |k|\right)^{2s}\left(1+| \sigma -  k^3|\right)^{2b}|\widetilde{u}\left(\sigma, k\right)|^2 \dd\sigma\right)^{\frac{1}{2}},
\end{equation*}
where $\widetilde{u}$ is the space-time Fourier transform of $u$:
$$
\widetilde{u}(\sigma, k)= \int_{\mathbb{R}\times \mathbb{T}} \e^{-i \sigma t - i k x} u(t,x) \, \dd t \dd x.
$$
We shall also use a localized version of this space.  For $I \subset \mathbb{R}$ being an open interval, we say that $u \in X^{s,b}(I)$ \black if $\|u\|_{X^{s,b}(I)}< \infty$, where
$$
\|u\|_{X^{s,b}(I)} = \inf\{\|\overline{u} \|_{X^{s,b}}, \, \overline{u}|_{I} = u  \}.
$$
When $I= (0, T)$ we will often simply use the notation $X^{s,b}(T)$. We refer for example to \cite{ORSBourg} Lem\-ma~2.1 for some useful properties of these spaces in this setting
(and to \cite{Bour93b} and \cite{Tao06} for more details).
A  particularly useful property  is the embedding $X^{s, b} \subset \mathcal{C}(\mathbb{R}, H^s)$ for $b>1/2$.
In the case of the KdV equation on the torus, in order to resolve the derivative in the nonlinearity, we are forced to work with the
borderline space which is at 
the level of $b=1/2$ (cf. \cite{KPV,CKSTT}). To get a space with good properties, we work with the smaller space
 $X^s$, which has the same scaling properties in time as $X^{s, {1 \over 2} }$,  defined by the following norm:
 \begin{equation}
 \label{Xsdef}
  \|u\|_{X^s} = \|u \|_{X^{s, {1 \over 2 }}} + \|\langle k \rangle^s \tilde u \|_{l^2(k)L^1(\sigma)}.
  \end{equation}
  We  define more precisely   $X^s$ as the  space of  space-time tempered distributions  such that $ \tilde{u}(\sigma, 0)= 0$ and
   the above norm is finite.
  In a similar way, we get a localized version $X^s(I)$ or $X^s(T)$ if $I=(0,T)$ by setting 
  $$ \|u\|_{X^{s}(I)} = \inf\{\|\overline{u} \|_{X^{s}}, \, \overline{u}|_{I} = u  \}.$$
  The main well-posedness result for \eqref{KdV} reads:
\begin{theorem}
\label{theoKdV}
For every $T>0$ and $u_{0}\in L^2$, $\int_{\mathbb{T}}u_{0}= 0$,  there exists a unique solution $u$ of \eqref{KdV} such that $u \in  X^0(T)$. Moreover, if $u_{0} \in H^{s_{0}}$, $s_{0}> 0$, then $ u\in X^{s_{0}}(T)$.
\end{theorem}
Note that we  have $X^0(T) \subset \mathcal{C}([0, T], L^2)$ and
$ X^{s_{0}}(T)\subset \mathcal{C}([0, T], H^{s_{0}})$. The result also holds true for initial data of some negative regularity, 
nevertheless, since we have choosen to measure the convergence of our  numerical schemes  in the natural  $L^2$ norm, 
we shall not use these more general results.

 We refer to \cite{KPV, Bour93b, CKSTT} for the  detailed proof, nevertheless, we shall recall the main ingredients since
 we will later use related arguments at the discrete level.  

\begin{proof}
The existence in short time is first  proven by a fixed point argument on the following truncated problem:
 $$v\mapsto \Phi(v)$$
 such that
 \begin{equation}
 \label{fixed} \Phi(v)(t)=   \chi(t)e^{- t \partial_{x}^3} u_{0} -  {1 \over 2} \chi(t)\int_{0}^t e^{-(t- s)\partial_{x}^3 } 
\partial_{x}\left( \chi\left({s \over \delta}\right)u(s)\right)^2\, ds,
\end{equation}
 where $\chi \in [0, 1]$ is a  smooth compactly supported function which is equal to $1$ on $[-1, 1]$
  and supported in $[-2, 2]$. For $|t| \leq \delta \leq 1/2 $,  a fixed point of the above equation
  gives a solution of the original Cauchy problem, denoted by $u$. The parameter $\delta >0$ is chosen
  in the proof to get a contraction.
  
 The basic  properties of the spaces $X^{s,b}$ and $X^s$ that are needed are the following:
  \begin{lemma}
 \label{lemmabourgainfacile}
 For $\eta\in \mathcal{C}^\infty_{c}(\mathbb{R})$,   we have that
 \begin{eqnarray}
 \label{bourg1}
&  \| \eta (t) e^{- t \partial_{x}^3 } f \|_{X^{s}} \lesssim_{\eta} \|f\|_{H^s}, \quad s \in \mathbb{R}, \, f \in H^s(\mathbb{T}), \\
  \label{bourg2}
   \label{bourg3}
&  \| \eta({t \over T}) u \|_{X^{s,b'}} \lesssim_{\eta, b, b'} T^{b-b'} \|u\|_{X^{s,b}}, \quad s \in \mathbb{R},  -{1 \over 2} <b' \leq b <{ 1 \over 2}, \, 0< T \leq 1, \\
\label{bourg4}
& \left\| \chi(t) \int_{0}^t e^{- (t-s) \partial_{x}^3} F(s) \, ds \right\|_{X^{s}} \lesssim_{\eta, b} \|F\|_{Y^{s}}, \quad s \in \mathbb{R}
 \end{eqnarray}
 where $Y^s$ is the space defined by the norm
 $$  \|F\|_{Y^s} = \|F \|_{X^{s, -{1 \over 2 }}} + \|\langle k \rangle^s \langle \sigma - k^3 \rangle^{-1} \tilde F \|_{l^2(k)L^1_{\sigma}}.$$
  \end{lemma}
  The other  ingredient is the following crucial  bilinear estimate.
  \begin{lemma}
  \label{bourgaindur}
  For $s \geq 0$, and $u \in X^s$,  we have the estimate:
 $$  \| \partial_{x} (u^2) \|_{Y^s} \lesssim  \|u\|_{X^{s, {1 \over 2}}} \|u \|_{X^{0, {1 \over 3}}} + \|u \|_{X^{0, {1 \over 2}}} \|u\|_{X^{s, {1 \over 3}}}.$$
  \end{lemma}
  The difficult part is to prove the estimate for $s=0$ (it actually holds true for $ s>-1/2$), it is afterwards easy to get
  the estimate for $s>0$. 
  Note that by definition of our space $X^s$, we have that $ \int_{\mathbb{T}}u(t, \cdot) = 0$.
  By polarization, we easily deduce that we also have for example
  \begin{equation}
  \label{polarbourg}
  \| \partial_{x} (u v)\|_{Y^0} \lesssim \| u \|_{X^{0, {1 \over 2}}} \|v \|_{X^{0, {1 \over 3 }}}
 + \| v \|_{X^{0, {1 \over 2}}} \|u \|_{X^{0, {1 \over 3 }}}.
   \end{equation}
  
  By using \eqref{bourg1} and \eqref{bourg4}, we get that for $v \in X^0$
  $$\| \Phi(v) \|_{X^0}  \leq C \left( \|u_{0}\|_{L^2}  + \left\| \partial_{x} \left(\chi\left({t  \over \delta}\right) v\right)^2 \right\|_{Y^0}\right),$$
  where $C>0$ is independent of $u_{0}$ and $\delta$. Then, we can use Lemma \ref{bourgaindur}, to get
  $$ \| \Phi(v) \|_{X^0} \leq C \left( \|u_{0}\|_{L^2}  + 
  \left\| \chi\left({t  \over \delta}\right) v\right\|_{X^{0, {1 \over 2}}}       \left\| \chi\left({t  \over \delta}\right) v \right\|_{X^{0, {1 \over 3}}}        \right)$$
and we finally deduce from \eqref{bourg2} that
$$    \| \Phi(v) \|_{X^0} \leq C \left( \|u_{0}\|_{L^2}  +  \delta^\epsilon \|v\|_{X^0}^2 \right),$$
where $\epsilon$ is any number in $ (0, 1/6).$ 
   By using the same ingredients, we also get that for every $v$, $w \in X^0$, we have that
   $$    \| \Phi(v)  - \Phi(w) \|_{X^0} \leq C  \delta^\epsilon  ( \|v\|_{X^0} + \| w \|_{X^0})\|v - w\|_{X^0}$$
   for some $C>0$ independent of $\delta$ and $u_{0}.$
   
   Consequently, by taking $R= 2 C \|u_{0}\|_{L^2}$, we get that there exists $\delta>0$ sufficiently small 
    that depends only on $\|u_{0}\|_{L^2}$, such that $\Phi$ is a contraction on the closed ball $B(0, R)$
     of $X^{0}.$ This proves the existence of a fixed point for $\Phi$ and hence the existence of  a solution $u$  of \eqref{KdV}
      on $[0, \delta]$.  
     By using again Lemma \ref{lemmabourgainfacile} and Lemma \ref{bourgaindur}, we also have for  $s \geq 0$, that  
    $$  \|\Phi(v) \|_{X^{s} } \leq C \|u_{0}\|_{H^s} + C \delta^\epsilon  \|v\|_{X^{0}} \|v\|_{X^{s}},$$
       such that if $u_{0}$ is in $H^s$ then we also have that  $u \in X^{s,b}( [0, \delta])$ .
      Since the $L^2$ norm is conserved for~\eqref{KdV}, we can  reiterate the construction on $[\delta, 2 \delta]$ and so on 
      to get a global solution. We thus obtain a solution  $u$ with  
         $u \in X^{s, b}(T)$  for every $T$.
        
\end{proof}
 
Let us now consider the projected equation \eqref{KdVP}. A straghtforward adaptation of the previous proof yields
 the following global well-posedness result.
 
 \begin{proposition}
\label{propKdVP}
For  $u_{0} \in H^{s_{0}}$, $s_{0} \geq 0$  and $\tau \in (0, 1]$, there exists a unique solution $u_{\tau}$ of \eqref{KdVP} such that $u_{\tau} \in X^{s_{0}}(T)$  for every $T >0$. Moreover, for every $T>0$, there exists $M_{T}> 0$ such that for every $\tau \in (0, 1]$, we have
 the estimate
 $$ \| u_{\tau} \|_{X^{s_{0}}(T)} \leq M_{T}.$$
\end{proposition}

\begin{rem}
\label{remlocal}
Note that, since  $\Pi_{\tau}^2= \Pi_{\tau}$, we have that $\Pi_{\tau} u_{\tau}$ solves the same equation \eqref{KdVP} with the same initial data as $u_{\tau}$ and hence we have by uniqueness that
$$
\Pi_{\tau} u_{\tau} (t) = u_{\tau}(t) \quad \text{for all \ $t \geq 0$.}
$$
\end{rem}

We shall also need an estimate with  more $b$ regularity:
\begin{cor}
\label{corbbig}
 For every $T\geq 1$ and $u_{0} \in H^{s_{0}},  $ $s_{0} \geq 0$, $\int_{\mathbb{T}}u_{0}= 0$,  there exists $M_{T}>0$ such that for every  $\tau \in (0, 1]$, we have the estimate
  $$ \| u_{\tau}\|_{X^{s_{0}, 1}(T)} \leq { M_{T} \over \tau^{1 \over 3}}.$$
\end{cor}
\begin{proof}
  For $\delta>0$, small enough (depending only on $T$ and $\|u_{0}\|_{H^{s_{0}}}$)  we have that $u_{\tau}$ coincides with the following  fixed point $U_{\tau}
   \in X^{s_{0}}$:
 $$ U_{\tau}(t) =   \chi(t)e^{- t \partial_{x}^3} \Pi_{\tau} u_{0} -  {1 \over 2} \chi(t)  \Pi_{\tau} \int_{0}^t e^{-(t- s)\partial_{x}^3 } 
\partial_{x} \left(  \Pi_{\tau}\chi\left({s \over \delta}\right)U_{\tau}(s)\right)^2\, ds,$$
  where $\chi \in [0, 1]$ is  a  smooth compactly supported function which is equal to $1$ on $[-1, 1]$
  and supported in $[-2, 2]$.  
We thus get that
$$ \|U_{\tau}\|_{X^{s_{0} 1}} \lesssim \|u_{0}\|_{L^2} +  \left\| \partial_{x} \Pi_{\tau} \left( \Pi_{\tau}\chi\left({t \over \delta}\right)U_{\tau} \right)^2
\right\|_{X^{s_{0}, 0}}.$$
Here we have used again \eqref{bourg1} and the following  general estimate for Bourgain spaces:
$$ \left\| \chi(t) \int_{0}^t e^{- (t-s) \partial_{x}^3} F(s) \, ds \right\|_{X^{s, {b}}} \lesssim_{\eta, b} \|F\|_{X^{s, {b-1}}}$$
for $b \in (1/2, 1]$.
 We refer again to \cite{KPV, Bour93b, CKSTT} and the book \cite{Tao06}. 
  To estimate the right hand side, we use that $\Pi_{\tau}$ projects on frequencies $|k| \leq \tau^{- {1 \over 3}}$.  Together with the generalized Leibniz rule 
this yields that
  $$  \|U_{\tau}\|_{X^{s_{0}, 1}} \lesssim \|u_{0}\|_{L^2} + { 1 \over \tau^{1 \over 3}} \left\| \langle \partial_{x} \rangle^{s_{0}} \Pi_{\tau}\chi\left({t \over \delta}\right)U_{\tau}
\right\|_{L^4(\mathbb{R}\times \mathbb{T})}^2.$$
To conclude, we use the Strichartz estimate for KdV on the torus (which is actually used for the proof
of Lemma \ref{bourgaindur}, we again refer to \cite{KPV, Bour93b, CKSTT}) which reads
$$ \| u \|_{L^4(\mathbb{R}\times \mathbb{T})} \lesssim \|u\|_{X^{0, {1 \over 3}}}$$
(we will prove   a discrete version of this estimate in Section
\ref{sectiontechnical}). This yields
 $$ \|U_{\tau}\|_{X^{s_{0}, 1}} \lesssim \|u_{0}\|_{L^2} + { 1 \over \tau^{1 \over 3}} \|U_{\tau}\|_{X^{s_{0}}}^2.$$ 
By iterating the argument, we thus deduce that 
 $$ \|u_{\tau}\|_{X^{s_{0}, {1}}(T)} \leq { M_{T} \over \tau^{1 \over 3}}$$
thanks to  Proposition \ref{propKdVP}.

\end{proof}

We can also easily get the following estimate on the difference $\|u (t)- u_{\tau}(t) \|_{L^2}$ which was the aim of this section.
\begin{proposition}
\label{corNLSK}
For  $u_{0} \in H^{s_{0}}$, $s_{0} \geq 0$, $\int_{\mathbb{T}} u_{0}=0$,  and every $T>0$, there exists $C_{T}>0$ such that for every $ \tau \in (0, 1],$
 we have the estimate
$$
\|u -u_{\tau}\|_{X^{0 }(T)} \leq  C_{T} \tau^{ s_{0}  \over 3}.
$$
\end{proposition}
Since $ X^0(T) \subset \mathcal{C}([0, T], L^2)$, we have in particular that
$$ \sup_{t \in [0, T]} \|u (t)-u_{\tau} (t)\|_{L^2} \leq  C_{T} \tau^{ s_{0}  \over 3}.$$

\begin{proof}
For some $\delta >0$ sufficiently small, we first observe that $u \in X^{s_{0}}(T)$ the solution of \eqref{KdV} coincides on $[0, \delta]$
with the fixed point of $\Phi$ defined in \eqref{fixed} which belongs to $X^{s_{0}}$. We shall (by abuse of notation) still denote by $u$
 this fixed point.
  In a similar way, $u_{\tau} \in   X^{s_{0}}(T)$ coincides on $[0, \delta]$
with the fixed point of $\Phi_{\tau}$ in $X^s_{0}$ that we shall still denote by $u_{\tau}$, where
$$ \Phi_{\tau}(v)(t)=   \chi(t)e^{- t \partial_{x}^3} \Pi_{\tau} u_{0} -  {1 \over 2} \chi(t) \Pi_{\tau}\int_{0}^t e^{-(t- s)\partial_{x}^3 } 
\partial_{x}\left(  \Pi_{\tau}\chi\left({s \over \delta}\right)v(s)\right)^2\, ds.$$
With these notations, we thus get that
\begin{multline} 
\label{estimationdiff}
u(t)- u_{\tau}(t)
=   \chi(t)e^{- t \partial_{x}^3} ( 1 -   \Pi_{\tau} )u_{0}
 - {1 \over 2}\chi(t) ( 1 - \Pi_{\tau}) \int_{0}^t e^{-(t- s)\partial_{x}^3 } 
\partial_{x}\left( \chi\left({s \over \delta}\right)u(s)\right)^2\, ds 
\\- {1 \over 2} \chi(t) \Pi_{\tau}   \int_{0}^t e^{-(t- s)\partial_{x}^3 } 
\partial_{x}\left( \chi\left({s \over \delta}\right) (1-\Pi_{\tau})u(s)\right)^2\, ds \\- 
   \chi(t) \Pi_{\tau}   \int_{0}^t e^{-(t- s)\partial_{x}^3 } 
\partial_{x}\left( \chi\left({s \over \delta}\right) (1-\Pi_{\tau})u(s)\,  \chi\left({s \over \delta}\right) \Pi_{\tau}u(s) \right)\, ds\\
- {1 \over 2}  \chi(t) \Pi_{\tau}   \int_{0}^t e^{-(t- s)\partial_{x}^3 } 
\partial_{x}\left(   \chi\left({s \over \delta}\right) \Pi_{\tau}(u(s)+  u_{\tau}(s))\chi\left({s \over \delta}\right) \Pi_{\tau}(u(s)- u_{\tau}(s))\right)\, ds.
\end{multline}
Thanks to the definition of $\Pi_{\tau}$, we have that 
$$ \| (1 - \Pi_{\tau}) f \|_{L^2} \leq \tau^{ s_{0} \over 3}\| f \|_{H^{s_{0}}}, \quad \forall f \in H^{s_{0}}$$
and thus 
$$ \| (1 - \Pi_{\tau}) f \|_{X^0} \leq \tau^{ s_{0} \over 3} \|f \|_{X^{s_{0}}}, \quad \forall f \in X^{s_{0}}.$$
Consequently, by using this observation and  again \eqref{bourg1}, \eqref{bourg2} as well as Lemma \ref{bourgaindur}, 
we obtain that
\begin{align*} \|u - u_{\tau}\|_{X^0}
& \lesssim  \tau^{s_{0} \over 3}  \|u_{0}\|_{H^{s_{0}}} +  \tau^{s_{0} \over 3} \left\|
  \chi(t) \int_{0}^t e^{-(t- s)\partial_{x}^3 } 
\partial_{x}\left( \chi\left({s \over \delta}\right)u(s)\right)^2\, ds \right\|_{X^{s_{0}}} \\&
 +  \| (1- \Pi_{\tau})u \|_{X^0}^2 + 2 \| (1- \Pi_{\tau}) u \|_{X^0} \|u\|_{X^0}
 + \delta^\epsilon( \|u \|_{X^0} + \|u_{\tau}\|_{X^0}) \|u- u_{\tau}\|_{X^0} \\ &
 \lesssim   \tau^{s_{0} \over 3}(\|u_{0}\|_{H^{s_{0}}} + \|u \|_{X^{s_{0}}}^2)
 +  \delta^\epsilon( \|u \|_{X^0} + \|u_{\tau}\|_{X^0}) \|u- u_{\tau}\|_{X^0}.
 \end{align*}
 Let us fix $M_{T}$ independent of $\tau \in (0, 1]$ and $ \delta \in (0, 1]$  such that
 $$ \| u\|_{X^{s_{0}} }+   \| u_{\tau}\|_{X^{s_{0}}} \leq M_{T}, $$
 we then obtain that
 $$  \|u - u_{\tau}\|_{X^0} \leq \tau^{s_{0} \over 3}(\|u_{0}\|_{H^{s_{0}}} + M_{T}^2)
 +  2 \delta^\epsilon M_{T} \|u- u_{\tau}\|_{X^0}.$$
  By taking $\delta$ sufficiently small so that   $2 \delta^\epsilon M_{T}<1/4$, we then
  obtain that
  $$   \|u - u_{\tau}\|_{X^0} \leq C_{T} \tau^{s_{0} \over 3}$$
  which gives the desired estimate on $[0, \delta]$. We
    can then iterate the argument to get the estimate on the full interval  $[0, T]$.

 \end{proof}
 
 \section{Discrete Bourgain-KdV spaces}
 
 \label{sectionBourgain}

In order to perform error estimates at low regularity, we shall develop at the discrete level
the harmonic analysis tools used in Section \ref{sec:errPro}.
Definitions and properties of discrete Bourgain spaces were introduced (in the context of the nonlinear Schr\"odinger
 equation) in \cite{ORSBourg}. Nevertheless, as in the continuous case, we need additional results
 in order to handle the KdV equation, namely we shall introduce the discrete counterpart of the space $X^s$, 
 study its properties 
  and prove a bilinear estimate analogous to the one of Lemma  \ref{bourgaindur}.
  
  For sequences of functions $(u^{n}(x))_{n \in \mathbb{Z}},$ we define the Fourier transform $\widetilde{u^{n}}(\sigma, k)$ by
$$
\mathcal F_{n,x}(u^n)(\sigma,k) =\widetilde{u^{n}} (\sigma, k)= \tau \sum_{m \in \mathbb{Z}} \widehat{u^{m}}(k) \,\e^{i m \tau \sigma}, \quad \widehat{u^{m}}(k)= {1 \over 2\pi} \int_{-\pi}^\pi u^{m}(x) \,\e^{-i k x}\dd x.
$$
Parseval's identity then reads
\begin{equation}\label{parseval}
\| \widetilde{u^{n}}\|_{L^2l^2}= \|u^{n}\|_{l^2_{\tau}L^2},
\end{equation}
where
$$
\| \widetilde{u^{n}}\|_{L^2l^2}^2 = \int_{-{\pi \over \tau}}^{\pi\over \tau} \sum_{k \in \mathbb{Z}}
|\widetilde{u^{n}}(\sigma, k)|^2 \dd \sigma, \quad
\|u^{n}\|_{l^2_{\tau}L^2}^2 = \tau \sum_{m \in \mathbb{Z}} \int_{-\pi}^\pi  |u^{m}(x)|^2 \dd x.
$$
We  define the discrete Bourgain spaces $X^{s,b}_\tau$ for $s\ge 0$, $b\in\mathbb R$, $\tau>0$ by
\begin{equation}\label{norm2}
\| u^n \|_{X^{s,b}_{\tau}} = \left\| \langle k \rangle^s \langle  d_{\tau}(\sigma + k^3)  \rangle^b \widetilde{u^n}(\sigma, k)  \right\|_{L^2l^2},
\end{equation}
where  $d_{\tau}(\sigma)=\frac{\e^{i \tau \sigma} - 1}\tau$.
Note that $d_{\tau}$ is $2\pi/\tau$ periodic and that uniformly in $\tau$, we have $|d_{\tau}(\sigma)| \sim | \sigma |$ for $|\tau \sigma | \leq \pi$. Since $|d_{\tau}(\sigma)| \lesssim \tau^{-1}$, we also have
%
  that the discrete spaces satisfy the embeddings
\begin{equation}\label{embdisc1}
\|u^{n}\|_{X^{0, b}_{\tau}} \lesssim { 1 \over  \tau^{b-b'}} \|u^{n}\|_{X^{0, b'}_\tau}, \quad b \geq b'.
\end{equation}
Some useful more technical properties are gathered in the following lemma:
\begin{lemma}\label{bourgainfaciled}
For $\eta \in \mathcal{C}^\infty_{c}(\mathbb{R})$ and $\tau\in(0,1]$, we have that
\begin{align}
\label{bourg1d} &\| \eta(n \tau)  \e^{- n \tau \partial_{x}^3} f\|_{X^{s,b}_{\tau}} \lesssim_{\eta, b} \|f\|_{H^s}, \quad s \in \mathbb{R}, \, b \in \mathbb{R}, \, f \in H^s, \\
\label{bourg2d} &\| \eta(n \tau)  u^{n}\|_{X^{s,b}_{\tau}} \lesssim_{\eta, b} \|u^{n}\|_{X^{s,b}_{\tau}}, \quad s \in \mathbb{R}, \, b \in \mathbb{R} , \, u^{n} \in X^{s,b}_{\tau},\\
\label{bourg3d} &\left\| \eta\left(\frac{n\tau}T \right) u^{n} \right\|_{X^{s,b'}_{\tau}} \lesssim_{\eta, b, b'} T^{b-b'} \|u^{n}\|_{X^{s,b}_{\tau}}, \quad s \in \mathbb{R},  -{1 \over 2} <b' \leq b <{ 1 \over 2},\, 0< T = N \tau  \leq 1, \, N \geq 1.
\end{align}
In addition, for
$$
U^{n}(x)= \eta(n \tau) \tau \sum_{m=0}^n  \e^{- ( n-m ) \tau \partial_{x}^3}  u^{m}(x),
$$
we have
\begin{equation}
\label{bourg4d}\|U^{n}\|_{X^{s,b}_{\tau}} \lesssim_{\eta, b} \|u^{n}\|_{X^{s, b-1}_{\tau}}, \quad s \in \mathbb{R}, \, b>1/2.
\end{equation}
We stress that all given estimates are uniform in $\tau$.
\end{lemma}
The proof directly follows from the ones of  \cite[Lemma 3.4]{ORSBourg}. Indeed, it suffices to observe that
$$ \| u^{n} \|_{X^{s,b}_{\tau}}= \| e^{ n \tau \partial_{x}^3 } u^{n}\|_{H^b_{\tau} H^s },$$
where 
$$ \| u^{n}\|_{H^b_{\tau}H^s} := \|   \langle d_{\tau}(\sigma)\rangle^b \langle k \rangle^s  \tilde u^{n} (\sigma, k)\|_{L^2 l^2 }$$
and the proofs only use the properties of the space $H^b_{\tau} H^s.$

The next step that we shall need in order to handle the KdV equation is to adapt  \eqref{bourg4} in the case 
$b=1/2$. We first define the discrete counterpart $X^s_{\tau}$ of the $X^s$ space.
We say that a sequence of function $(u^{n}(x))_{n} \in l^2_{\tau}L^2$ such that 
$ \int_{\mathbb{T}} u^n=0, \, \forall n$ is in $X^s_{\tau}$ for $s \geq 0$ if 
the following norm is finite
$$ \|u^{n}\|_{X^s_{\tau}}= \|u^{n}\|_{X^{s, {1 \over 2}}_{\tau}} + \|\langle k \rangle^s \tilde u( \sigma, k) \|_{l^2(k)L^1(\sigma)}$$
and in the same way,  we also define $Y_{\tau}^s$ by 
 $$ \|F^{n}\|_{Y^s_{\tau}}= \|F^{n}\|_{X^{s, -{1 \over 2}}_{\tau}}  +  \left\|{\langle k \rangle^s \over \langle d_{\tau}( \sigma + k^3\rangle)} \widetilde{F^{ n}}(\sigma, k) \right\|_{l^2(k)L^1(\sigma)}.$$

 \begin{lemma}
 \label{lemb1/2}
  We have the following properties:
 \begin{enumerate}
 \item We have the embedding $X^s_{\tau} \subset l^\infty(\mathbb{Z}, H^s(\mathbb{T}))$:
 \begin{equation}
 \label{sob1/2}
  \sup_{n} \| u^{n}\|_{H^s(\mathbb{T})} \lesssim \|u^{n} \|_{X^s_{\tau}}, \quad s \in \mathbb{R}, \, (u^{n})_{n} \in X^s_{\tau};
  \end{equation}
 \item  Let us define for $(u^{n})_{n} \in Y^s_{\tau}$,  and $\eta \in \mathcal{C}^\infty_{c}(\mathbb{R})$
 \begin{equation}
 \label{duhabourg}
U^{n}(x): = \eta(n \tau) \tau \sum_{m=0}^n  \e^{- ( n-m ) \tau \partial_{x}^3}  u^{m}(x),
\end{equation}
then, we have
\begin{equation}
\label{bourg1/2}\|U^{n}\|_{X^{s}_{\tau}} \lesssim_{\eta} \|u^{n}\|_{Y^{s}_{\tau}}, \quad s \in \mathbb{R}.
\end{equation}
 \end{enumerate}
 The above estimates are uniform for $\tau \in (0, 1]$.
 \end{lemma}
 
 \begin{proof} 
 We first prove \eqref{sob1/2}. By definition of our Fourier transforms, we have that for every $k \in \mathbb{Z}$,
  and every $m \in \mathbb{Z}$, we have
 $$ \widehat {u^{m}} (k)= { 1 \over 2 \pi} \int_{- {\pi \over \tau}}^{\pi \over \tau} \widetilde{u^{n}}(\sigma, k) e^{-i m \tau \sigma} \, d \sigma$$
  and hence 
$$ | \widehat {u^{m}} (k)| \leq  { 1 \over 2\pi}\| \widetilde{u^{n}}( \cdot, k) \|_{L^1(\sigma)}.$$
 Consequently, by taking the $l^2$ norm in $k$ and by using the Bessel identity, we obtain
 $$ \| u^{m}\|_{L^2(\mathbb{T})} \lesssim \| \widetilde{u^{n}}( \cdot, k) \|_{l^2(k)L^1(\sigma)} \leq \| 
 { {u^{n}}} \|_{X^0_{\tau}}.$$
 This gives \eqref{sob1/2} for $s=0$, and the general case follows by replacing $u^{n}$ by $\langle \partial_{x} \rangle^s u^{n}.$
 
 Let us now prove \eqref{bourg1/2}. Again, we give the proof for $s=0$, the general case just follows by applying
 $ \langle \partial_{x}\rangle^s$ to the two sides of \eqref{duhabourg}.
  Let us set
 $$ F^{n}(x)= e^{+n \tau \partial_{x}^3} U^{n}(x), \quad f^{n}(x) = e^{+n \tau \partial_{x}^3} u^{n}(x)$$
  so that 
  $$ F^{n}(x)= \eta(n\tau)  \tau \sum_{m=0}^n  f^{m}.$$
 We shall first  prove that
  \begin{equation}
  \label{variant1} \|F^{n}\|_{H^{1 \over 2}_{\tau}L^2} + \|\widetilde{F^n} \|_{l^2(k)L^1(\sigma)} \lesssim \| f^{n}\|_{H^{1 \over 2}_{\tau}L^2} + \left\|  { 1 \over \langle d_{\tau}(\sigma) \rangle}
   \widetilde{f^n}\right\|_{l^2(k) L^1(\sigma)}
   \end{equation}
   which is equivalent to 
   $$ \|{ U}^n\|_{X^{0, {1 \over 2}}_{\tau}} + \|\widetilde{{U}^n}\|_{l^2(k)L^1(\sigma)} \lesssim \|u^n \|_{Y^0}.$$
   By direct computation, we find that 
   $$ \widetilde{F^{n}}(\sigma, k)=  { 1 \over 2\pi} \int_{-{\pi \over \tau}}^{\pi \over \tau} {e^{i \tau \sigma_{0}} \over d_{\tau}(\sigma_{0})} \widetilde{f^{n}}(\sigma_{0}, k)
    \left(g(\sigma)- e^{-i \tau \sigma_{0}} g(\sigma - \sigma_{0})\right)\, d \sigma_{0},$$
    where $g(\sigma)=  \mathcal{F}_{\tau}( \eta(n \tau))(\sigma)$. Note that $g$   is fastly decreasing in the sense that
   \begin{equation}
   \label{gfast}\left| \left\langle  d_{\tau}(\sigma)\right\rangle^K g(\sigma) \right| \lesssim 1,
   \end{equation}
 where   the estimate is uniform  in  $\tau\in (0, 1]$  and $\sigma $ for every $K$. We then 
 split, 
 $$  \widetilde{F^{n}}(\sigma, k)= \int_{|\sigma_{0}| \leq 1} + \int_{| \sigma _{0}| \geq 1}:= \widetilde{F^{n}_{1}}
  +  \widetilde{F^{n}_{2}} .$$
  For $|\sigma_{0}| \leq 1$, we can use the Taylor formula and the fast decay of $g$ to get that
  $$ \left|  { \langle d_{\tau}(\sigma) \rangle^{ 1 \over 2} \over d_{\tau}(\sigma_{0})} \left(g(\sigma)- e^{-i \tau \sigma_{0}} g(\sigma - \sigma_{0})\right)  \right|
   \lesssim { 1 \over \langle d_{\tau}(\sigma - \sigma_{0}) \rangle^K}.$$
   Therefore, we obtain that
   $$ \langle d_{\tau}(\sigma) \rangle^{ 1 \over 2} | \widetilde{F^{n}_{1}}|
    \lesssim \int_{| \sigma _{0}| \leq 1}  {  1 \over \langle d_{\tau}(\sigma - \sigma_{0})\rangle^K} | \widetilde{f^n}(\sigma_{0}, k) | d \sigma_{0}.$$
     By choosing $K$ large enough  we thus find that
    \begin{multline}
    \label{Fn1}
     \| F^n_{1} \|_{H^{1 \over 2}_{\tau}L^2} + \| \widetilde{F^n}\|_{l^2(k)L^1(\sigma)} \\ \lesssim \int_{| \sigma_{0}| \leq 1} \| \widetilde{f^n}(\sigma_{0}, \cdot)\|_{L^2} + \left\|  \int_{| \sigma_{0}| \leq 1} |\widetilde{f}^n|(\sigma_{0},  k)\, d\sigma_{0}  \right\|_{l^2(k)}
      \lesssim \| f^n \|_{H^{- {1 \over 2}}_{\tau}L^2} +\left\|  { 1 \over \langle d_{\tau}(\sigma) \rangle}
   \widetilde{f^n}\right\|_{l^2(k) L^1(\sigma)} ,
   \end{multline}
      where we have used that $| \sigma_{0}| \leq 1$ and  Cauchy-Schwarz (for the first term).
      
      It remains to bound the second term $F^n_{2}$. We write
  $$  \langle d_{\tau}(\sigma) \rangle^{ 1 \over 2} | \widetilde{F^{n}_{2}}|
   \lesssim \int_{| \sigma_{0}| \geq 1}  { 1 \over \langle d_{\tau}(\sigma) \rangle^K} {1 \over  \langle d_{\tau}(\sigma_{0}) \rangle }
    | \widetilde{f^n}(\sigma_{0}, k) | d \sigma_{0} +   \int_{| \sigma_{0}| \geq 1}  {1 \over  \langle d_{\tau}(\sigma_{0}) \rangle^{1 \over 2} }
    | \widetilde{f^n}(\sigma_{0}, k) | { 1 \over \langle d_{\tau}(\sigma - \sigma_{0}) \rangle^K } d \sigma_{0},$$
   where we have used that
   $\langle d_{\tau} (\sigma ) \rangle^{1 \over 2} | g(\sigma)| \lesssim \langle d_{\tau} (\sigma) \rangle^{-K}$  for the first term 
   and $  \langle d_{\tau} (\sigma ) \rangle^{1 \over 2} |g(\sigma- \sigma_{0}) | \leq  \langle d_{\tau}(\sigma_{0})\rangle^{1 \over 2}
    \langle d_{\tau}( \sigma - \sigma_{0} )\rangle^{-K}$
   for the second one with $K$ large enough. By taking the $L^2$ norm in $\sigma$ and by using the Young inequality for convolutions for the second
   term, we get that
  $$ \| \langle d_{\tau}(\sigma) \rangle^{ 1 \over 2}  \widetilde{F^{n}_{2}}(\cdot, k) \|_{L^2(\sigma)}
    \lesssim  \left \| {1 \over  \langle d_{\tau} \rangle }
     \widetilde{f^n}(\cdot, k)   \right\|_{L^1(\sigma)} + \left\|
     {1 \over  \langle d_{\tau} \rangle^{1 \over 2} }
     \widetilde{f^n}(\cdot, k)  \right\|_{L^2(\sigma)}.$$
     Finally, by taking the $l^2$ norm in $k$, we obtain that
   \begin{equation}
   \label{Fn21} \| \widetilde{F^{n}_{2}}\|_{H^{1 \over 2}_{\tau}l^2(k)}
    \lesssim  \left\|  { 1 \over \langle d_{\tau}(\sigma) \rangle}
   \widetilde{f^n}\right\|_{l^2(k) L^1(\sigma)} + \|f^n \|_{H^{- {1 \over 2}}_{\tau}l^2(k)}.
   \end{equation}
   From the fast decay of $g$, we also have by similar arguments that
   $$   |\widetilde{F^{n}_{2}}(\sigma, k)|
   \lesssim \int_{| \sigma_{0}| \geq 1}  { 1 \over \langle d_{\tau}(\sigma) \rangle^K} {1 \over  \langle d_{\tau}(\sigma_{0}) \rangle }
    | \widetilde{f^n}(\sigma_{0}, k) | d \sigma_{0} +   \int_{| \sigma_{0}| \geq 1}  {1 \over  \langle d_{\tau}(\sigma_{0}) \rangle }   | \widetilde{f^n}(\sigma_{0}, k) | { 1 \over \langle d_{\tau}(\sigma - \sigma_{0}) \rangle^K } d \sigma_{0}.$$
    By taking the $L^1$ norm in $\sigma$ and then the $l^2$ norm in $k$, we thus find that
    \begin{equation}
    \label{Fn22}   \|\widetilde{F^{n}_{2}}(\sigma, k)\|_{l^2(k) L^1(\sigma)} \lesssim \left\|  {1 \over \langle d_{\tau}(\sigma) \rangle} \widetilde{f^n} \right\|_{l^2(k)L^1(\sigma)}.
    \end{equation}
    
  Gathering \eqref{Fn1}, \eqref{Fn21} and \eqref{Fn22}, we finally get  \eqref{variant1}, this  ends the proof of \eqref{bourg1/2}.

 \end{proof}

 The next result, we will need is the discrete counterpart of Lemma \ref{bourgaindur}:
 \begin{lemma}
 \label{bourgaindurd}
 For every $s \geq 0$, there exists $C>0$ such that for every $(u^n)_{n}$, $(v^n)_{n} \in X^{s}_{\tau}$, we have the estimate
 $$ \| \partial_{x} \Pi_{\tau} \left( \Pi_{\tau} u^n\,  \Pi_{\tau}v^n \right) \|_{Y^{s}_{\tau}}
  \leq C \left(\|u^n \|_{X^{s, {1 \over 2} }_{\tau}} \|v^n \|_{X_{\tau}^{s, {1 \over 3}}} + \|v^n\|_{X^{s, {1 \over 2} }_{\tau}} \|u^n \|_{X_{\tau}^{s, {1 \over 3}}}
  \right).$$
 \end{lemma}
 
 Note that as in the continuous case, the above estimate does not involve additional space derivatives in the right hand-side.
  The use of the projections $\Pi_{\tau}$ is crucial to get this property. Since the understanding of the proof of this lemma is not
  essential to understand the error estimates, we postpone it to Section \ref{sectiontechnical}.
  
  The last property we shall need is to relate the discrete and the  continuous Bourgain norms for  the sequence
   defined by $u^{n}=u_{\tau}(t_{n})$ where $u_{\tau}$ is the solution of \eqref{KdVP} given by Proposition \ref{propKdVP}.
    We shall still denote by $u_{\tau}$ an extension of $u_{\tau} \in X^{s_{0}}$ which coincides with 
     $u_{\tau}$ on $[- 4T , 4T ]$ and such that thanks to Proposition \ref{propKdVP} and Corollary \ref{corbbig}
     \begin{equation}
     \label{extutau}
   \|u_{\tau}\|_{X^{s_{0}}} + \tau^{1 \over 3}   \|u_{\tau}\|_{X^{s_{0}, 1}} \leq M_{T}
   \end{equation}
   for some $M_{T}$ independent of $\tau\in (0, 1]$.
   
   \begin{lemma}
   \label{lemmadisc-cont}
   Let  $T \geq 1$  and  let  $u_{\tau}$ be an extension  as above of  the solution of \eqref{KdVP} given by Proposition \ref{propKdVP}. 
    Then, there exists $C_{T}>0$ such that for every $\tau \in (0, 1]$, 
  we have the estimate
  $$ \sup_{s \in [-4 \tau, 4 \tau]}  \left\|u_{\tau}(t_{n}+ s) \right\|_{X^{s_{0}, {1 \over 2}}_{\tau}} 
   \leq C _{T}.$$
   
   \end{lemma}
   
   \begin{proof}
  Let us set 
   $f= \langle \partial_{x} \rangle^{s_{0}}\e^{-it \partial_{x}^2} u_{\tau}( \cdot + s)$ and $f^{n}(x)= f(n\tau, x)$, it suffices to prove that
$$
\|f^{n}\|_{H^{1 \over 2}_{\tau}L^2} \lesssim \|f\|_{H^{1 \over 2}L^2} + \tau^{1 \over 2} \| f \|_{H^1L^2}.
$$
Then we can conclude from \eqref{extutau}.

The discrete Fourier transform of the sequence $(f_{m})_{m}$ is by definition given by
$$
\widetilde{f^m}(\sigma, k)= \tau \sum_{n\in\mathbb Z} \widetilde f(n\tau, k) \,\e^{in \tau \sigma}.
$$
We thus  have by Poisson's summation formula that
$$
\widetilde{f^n}(\sigma, k)= \sum_{m \in \mathbb{Z}} \widetilde f\Bigl( \sigma + { 2 \pi \over \tau} m, k\Bigr), \quad
 \sigma \in [- \pi/\tau, \pi/\tau].
$$
Therefore,
$$
\langle d_{\tau}(\sigma) \rangle^{1 \over 2} \widetilde{f^n}(\sigma, k) = d_{\tau}(\sigma) \tilde f(\sigma, k) + \sum_{m  \in \mathbb{Z}, \, m \neq 0} \left\langle d_{\tau}(\sigma) \right\rangle^{1 \over 2} \widetilde f\Bigl( \sigma + { 2 \pi \over \tau} m, k\Bigr).
$$
 Since, we have   $|d_{\tau}(\sigma)| \lesssim \langle \sigma \rangle$, we get 
from  Cauchy--Schwarz that 
\begin{multline*}
|\langle d_{\tau}(\sigma) \rangle^{1 \over 2} \widetilde{f^n}(\sigma, k)|^2 \lesssim
 | \langle \sigma \rangle^{1 \over 2} \tilde f (\sigma, k )|^2 + 
 \sum_{\mu \neq 0} { 1 \over \left\langle \sigma +  { 2 \pi \over \tau} \mu \right \rangle^{2} }
\sum_{m \in \mathbb{Z}, \, m \neq 0} \left\langle \sigma + { 2 \pi \over \tau} m \right\rangle^{2} \langle d_{\tau}(\sigma)\rangle\left| \widetilde f\Bigl( \sigma + { 2 \pi \over \tau} m, k\Bigr)\right|^2.
\end{multline*}
We then observe that for $\sigma \in [- \pi/\tau, \pi/\tau]$, $\mu  \neq 0$ we have that
$$ \left | \sigma + {2\pi \mu \over \tau }\right| \geq {\pi  |\mu| \over \tau} $$ so
that
$$  \sum_{\mu \neq 0} { 1 \over \left\langle \sigma +  { 2 \pi \over \tau} \mu \right \rangle^{2} }
 \lesssim \tau^2.$$
 By using that $|d_{\tau}(\sigma)| \leq 2/ \tau$, we thus find that
$$ |\langle d_{\tau}(\sigma) \rangle^{1 \over 2} \widetilde{f^n}(\sigma, k)|^2 \lesssim
 | \langle \sigma \rangle^{1 \over 2} \tilde f (\sigma, k )|^2 +  \tau 
 \sum_{m \in \mathbb{Z}, \, m \neq 0} \left\langle \sigma + { 2 \pi \over \tau} m \right\rangle^{2}\left| \widetilde f\Bigl( \sigma + { 2 \pi \over \tau} m, k\Bigr)\right|^2.$$
  By integrating with respect to $\sigma \in [-\pi/\tau, \pi/\tau]$ and summing over $k$, we thus  obtain that
$$
\| f^{n}\|_{H^{1 \over 2}_{\tau}L^2} \lesssim
\|f \|_{H^{1 \over 2}L^2} + \tau^{1 \over 2}
 \| f \|_{H^1L^2}.
$$
This ends the proof.
\end{proof}

\section{Error estimate of the time discretisation of the modified projected equation}\label{sec:err}
In this section we derive an estimate on    the time discretisation error introduced by the discretisation~\eqref{scheme} applied to the  projected equation~\eqref{KdVP}. This will give an estimate on
$$
\Vert \uta(t_n) - u^n\Vert_{L^2}.
$$
Let us denote by $\Phi^\tau$ the numerical flow of \eqref{scheme} and by $\varphi_\tau^t$ the exact flow of the projected KdV equation \eqref{KdV}. Then we have
\[
\varphi_\tau^t(\uta(t_n)) = \uta(t_n+t) \quad \text{and} \quad u^{n+1} = \Phi^\tau(u^n).
\]
The mild solution of the projected KdV equation  \eqref{KdVP} is given by Duhamel's formula
\begin{equation}\label{exact}
\uta(t_n+\tau) = \varphi_\tau^\tau(\uta(t_n)) = e^{ - \tau \partial_x^3} \uta(t_n) - \frac12 e^{ - \tau \partial_x^3} \int_0^\tau 
e^{ s \partial_x^3}
\Pi_{\taul} \partial_x \left(  \Pi_{\taul} \uta(t_n+s)\right) ^2
ds.
\end{equation}
With the aid of the notation \eqref{Psi} we can furthermore express the numerical flow $\Phi^\tau$ applied to some function $v$ as follows
\begin{equation}\label{numi}
\Phi^\tau(v)  = \mathrm{e}^{-\tau \partial_x^3}  v
- \frac12 
 \mathrm{e}^{-\tau \partial_x^3}  \int_0^\tau 
\psi_1(s,\partial_x)
\Pi_{\taul} \partial_x \left(  \Pi_{\taul} \psi_2(s,\partial_x)v \right) ^2
ds.
\end{equation}
\subsection{Local error analysis}
 Taking the difference between \eqref{exact} and \eqref{numi} we see (by iterating Duhamels formula replacing $\tau$ by $s$ in \eqref{exact}) that the local error 
 $$
 \mathcal{E}(\tau,t_n) =  e^{ \tau  \partial_x^3}\left( \varphi_\tau^\tau(\uta(t_n)) - \Phi^\tau(\uta(t_n))\right)
 $$  takes the form
\begin{equation}\label{loc}
\begin{aligned}
& \mathcal{E}(\tau,t_n)\\
& = 
-   \frac12 \int_0^\tau 
\left[e^{ s \partial_x^3} - \psi_1(s,\partial_x)\right]
\Pi_{\taul} \partial_x \left(  \Pi_{\taul} \psi_2(s,\partial_x) \uta(t_n)\right) ^2
ds\\
&
 - \frac12 \int_0^\tau 
e^{ s \partial_x^3}
\Pi_{\taul} \partial_x \Big(
\left[  \Pi_{\taul} \big( \uta(t_n+s) -\psi_2(s,\partial_x)  \uta(t_n)\big) \right]
  \Pi_{\taul} \left[  \uta(t_n+s) +\psi_2(s,\partial_x)  \uta(t_n) \right]
\Big)
ds\\
& = 
- \frac12 \int_0^\tau 
\left[e^{ s \partial_x^3} - \psi_1(s,\partial_x)\right]
\Pi_{\taul} \partial_x \left(  \Pi_{\taul} \psi_2(s,\partial_x) \uta(t_n)\right) ^2
ds\\
&
 -  \frac12 \int_0^\tau 
e^{ s \partial_x^3}
\Pi_{\taul} \partial_x \Big(
\left[  \Pi_{\taul} \big( e^{s\partial_x^3} -\psi_2(s,\partial_x) \big) \uta(t_n) \right]  
  \Pi_{\taul} \left[  \uta(t_n+s) +\psi_2(s,\partial_x)  \uta(t_n) \right]
\Big)
ds\\
& +  \frac14 \int_0^\tau 
e^{ s \partial_x^3}
\Pi_{\taul} \partial_x \Big(
\left[ \int_0^\tau 
e^{ - (s-\xi ) \partial_x^3}
\Pi_{\taul} \partial_x \left(  \Pi_{\taul} \uta(t_n+\xi)\right) ^2
d\xi\right]  
  \Pi_{\taul} \left[  \uta(t_n+s) +\psi_2(s,\partial_x)  \uta(t_n) \right]
\Big)
ds.
\end{aligned}
\end{equation}

\subsection{Global error analysis}
Let $e^{n+1} = \uta(t_{n+1}) - u^{n+1}$ denote the time discretisation error. By inserting zero in terms of $\pm\Phi^\tau(\uta(t_n))$ we obtain
 $$
\begin{aligned}
e^{n+1} & = \varphi^\tau_\tau(\uta(t_n)) - \Phi^\tau(u^n)\\
& =   \varphi^\tau_\tau(\uta(t_n))- \Phi^\tau(\uta(t_n)) + \Phi^\tau(\uta(t_n))- \Phi^\tau(u^n)\\
& =
\mathrm{e}^{-\tau \partial_x^3}  e^n + \mathcal{J}^\tau(e^n, \uta(t_n) )+ e^{ -\tau  \partial_x^3} \mathcal{E}(\tau,t_n) \\
& =\sum_{\ell =0}^{n }
\mathrm{e}^{-(n-\ell) \tau \partial_x^3} \mathcal{J}^\tau(e^\ell, \uta(t_\ell) )+\sum_{\ell =0}^{n }
\mathrm{e}^{-(n-\ell+1) \tau \partial_x^3}    \mathcal{E}(\tau,t_\ell),
\end{aligned}
$$
where
\begin{equation}\label{J}
\begin{aligned}
&\mathcal{J}^\tau(e^n, \uta(t_n)) := - \frac12 
 \mathrm{e}^{-\tau \partial_x^3}  \int_0^\tau 
\psi_1(s,\partial_x)
\Pi_{\taul} \partial_x \Big[\left(  \Pi_{\taul} \psi_2(s,\partial_x)e^n
\right) \left(  \Pi_{\taul} \psi_2(s,\partial_x)(-e^n+2\uta(t_n)  )
\right) 
\Big]
ds
\end{aligned}
\end{equation}
and  the local error $\mathcal{E}(\tau,t_n)=  e^{ \tau  \partial_x^3}\left( \varphi_\tau^\tau(\uta(t_n)) - \Phi^\tau(\uta(t_n))\right)$ is  given
by
\begin{equation}\label{locErr}
\mathcal{E}(\tau,t_n) =  \sum_{j=1}^3\mathcal{E}^\tau_j(t_n)\end{equation}
with (see  \eqref{loc})
\begin{align}
\label{defE1}
\mathcal{E}^\tau_1(t_n)  & = -  \frac12 \int_0^\tau 
\left[e^{ s \partial_x^3} - \psi_1(s,\partial_x)\right]
\Pi_{\taul} \partial_x \left(  \Pi_{\taul} \psi_2(s,\partial_x) \uta(t_n)\right) ^2
ds
\\
\label{defE2}
\mathcal{E}^\tau_2(t_n)  & = - \frac12 \int_0^\tau 
e^{ s \partial_x^3}
\Pi_{\taul} \partial_x \Big(
\left[  \Pi_{\taul} \big( e^{s\partial_x^3} -\psi_2(s,\partial_x) \big) \uta(t_n) \right]  
  \Pi_{\taul} \left[  \uta(t_n+s) +\psi_2(s,\partial_x)  \uta(t_n) \right]
\Big)
ds
\\
\label{defE3}
\mathcal{E}^\tau_3(t_n)  & =  \frac14 \int_0^\tau 
e^{ s \partial_x^3}
\Pi_{\taul} \partial_x \Big(
\left[ \int_0^\tau 
e^{ - (s-\xi ) \partial_x^3}
\Pi_{\taul} \partial_x \left(  \Pi_{\taul} \uta(t_n+\xi)\right) ^2
d\xi\right]  
  \Pi_{\taul} \left[  \uta(t_n+s) +\psi_2(s,\partial_x)  \uta(t_n) \right]
\Big)
ds.
\end{align}
In order to use global Bourgain spaces, we use again
$\eta$  a smooth and compactly supported function, which is one on $[-1, 1]$ and supported in $[-2, 2]$ 
and we consider  $e^n$  that will solve  for $n \in \mathbb{Z}$ the following fixed point:
\begin{equation}
\label{enmod}
e^{n+1}= \eta(t_{n}) \sum_{\ell =0}^{n }
\mathrm{e}^{-(n-\ell) \tau \partial_x^3} \mathcal{J}^\tau\left(  \eta\left({t_{\ell} \over T_{1}}\right)e^\ell, \eta\left({t_{\ell} \over T_{1}}\right)\uta(t_\ell) \right)+ \eta(t_{n})\sum_{\ell =0}^{n }
\mathrm{e}^{-(n-\ell+1) \tau \partial_x^3}  \eta\left(t_{\ell} \right)    \mathcal{E}(\tau,t_\ell),
\end{equation}
where $\mathcal{J}^\tau$ and $\mathcal{E}$ are now defined by \eqref{J}, \eqref{locErr} with $u_{\tau}$
replaced by a global extension satisfying the estimate \eqref{extutau} and $T_{1}>0$, $T_{1}  \leq 1 \leq  T$
 will be chosen sufficiently small.
 We observe that for $0 \leq n \leq N_{1}$, where $N_{1}=\lfloor {T_{1}\over \tau}\rfloor$, a solution of the above fixed
  point coincides with $u_{\tau}(t_{n})- u^n$.
  
  With these new definitions, we have the following estimate on the global error:
  
  \begin{proposition}
  \label{propglobal}
 There exists $C_{T}>0$ such that for every $\tau \in (0, 1]$, we have the estimate
 $$ \tau^{-1} \| \mathcal{E}(\tau, t_{n}) \|_{Y^0_\tau} \leq C_{T} \tau^{\alpha}, \quad \alpha = \min \left( 1, {s_{0}\over 3}\right).$$
  \end{proposition}
  
  \begin{proof}
  Thanks to \eqref{locErr}, we estimate each of the $\mathcal{E}_{j}^\tau(t_{n})$.
  
  For  $\mathcal{E}_{1}^\tau(t_{n})$, 
   since $\e^{s \partial_{x}^3}- \psi_{i}$, $i=1, \, 2$ and $\Pi_{\taul}$ are Fourier multipliers
   in the space variable, and since  $\Pi_{\taul}$ projects on frequencies less than $\tau^{- {1 \over 3}},$ we observe that
  for any function $(F(t_{n}))_{}$, we have by Taylor expansion that 
  \begin{equation}
  \label{gainerreur} \sup_{s \in [- \tau, \tau]}\| \left(\e^{s\partial_{x}^3}-\psi_{i}(s, \partial_{x})\right) \Pi_{\tau} F(t_{n}) \|_{Y^{0}_{\tau}}
   \lesssim \tau^{\alpha}  \| F(t_{n}) \|_{Y^{s_{0}}_{\tau}}.
   \end{equation}
  
   Therefore, we get that
  $$ \tau^{-1} \| \mathcal{E}^\tau_{1}(t_{n}) \|_{Y^0_\tau} 
   \leq \tau^{\alpha}\sup_{s \in [0, \tau]} \|\Pi_{\tau}  \partial_x \left(  \Pi_{\taul} \psi_2(s,\partial_x) \uta(t_n)\right) ^2\|_{Y^{s_{0}}_{\tau}}.$$
Then by using {Lemma \ref{bourgaindurd} }and the fact that $\Pi_{\tau}$, $\psi_{2}$ are bounded  Fourier multiplier (in space), 
we get that
$$ \tau^{-1} \| \mathcal{E}^\tau_{1}( t_{n}) \|_{Y^0_\tau}  \lesssim  \tau^{\alpha}\|u_{\tau}\|_{X^{s_{0}, {1 \over 2}}_{\tau}}^2.$$
By using Lemma \ref{lemmadisc-cont}, we finally get that
$$  \tau^{-1} \| \mathcal{E}^\tau_{1}( t_{n}) \|_{Y^0_\tau} \leq C_{T} \tau^{\alpha}.$$

 For $\mathcal{E}^\tau_{2}$ defined in \eqref{defE2}, by using  again Lemma \ref{bourgaindurd} we get that
 $$  \tau^{-1} \| \mathcal{E}^\tau_{2}(t_{n}) \|_{Y^0_\tau} 
  \leq \sup_{s \in [0, \tau]}\left\|   
   \Pi_{\taul} \big( e^{s\partial_x^3} -\psi_2(s,\partial_x) \big) \uta(t_n) \right\|_{X^{0, {1 \over 2}}_{\tau}} \left\|
  \Pi_{\taul} \left[  \uta(t_n+s) +\psi_2(s,\partial_x)  \uta(t_n) \right] \right\|_{X^{0, {1 \over 2}}_{\tau}}.$$
  Consequently, by using again  \eqref{gainerreur} and Lemma \ref{lemmadisc-cont}, we find again that
  $$  \tau^{-1} \| \mathcal{E}^\tau_{2}(t_{n}) \|_{Y^0_\tau}  \lesssim C_{T} \tau^{\alpha}.$$
 In a similar way, for $\mathcal{E}_{3}^\tau$ defined in \eqref{defE3}, we first use Lemma \ref{bourgaindurd}
  to get that
  \begin{multline}
  \label{debutE2}
   \tau^{-1} \| \mathcal{E}^\tau_{3}(t_{n}) \|_{Y^0_\tau} 
  \\ \lesssim  \sup_{s \in [0, \tau]} \left(\left\|   
\left[ \int_0^\tau 
e^{ - (s-\xi ) \partial_x^3}
\Pi_{\taul} \partial_x \left(  \Pi_{\taul} \uta(t_n+\xi)\right) ^2
d\xi\right]  \right\|_{X^{0, {1 \over 2}}_{\tau}}  \left\|
  \Pi_{\taul} \left[  \uta(t_n+s) +\psi_2(s,\partial_x)  \uta(t_n) \right] \right\|_{X^{0, {1 \over 2}}_{\tau}} \right).
  \end{multline}
  By using again the property of $\Pi_{\tau}$, we first write that  in the case $s_{0}\leq 1$, we have
  $$ \tau^{-1} \| \mathcal{E}^\tau_{3}(t_{n}) \|_{Y^0_\tau}
  \leq \tau   \tau^{- {1 - s_{0} \over 3}} \sup_{s, \, \xi \in [0, \tau]}  \|\left(\Pi_{\taul} \uta(t_n+\xi)\right) ^2\|_{X^{s_{0}, {1 \over 2}}_{\tau}}
   \left\|
  \Pi_{\taul} \left[  \uta(t_n+s) +\psi_2(s,\partial_x)  \uta(t_n) \right] \right\|_{X^{0, {1 \over 2}}_{\tau}} .$$
  Next, thanks to the property \eqref{embdisc1} of discrete Bourgain spaces, we get that
  $$ \tau^{-1} \| \mathcal{E}^\tau_{3}(t_{n}) \|_{Y^0_\tau}
  \leq \tau^{1 \over 2}  \tau^{- {1 - s_{0} \over 3}} \sup_{s, \, \xi \in [0, \tau]}  \|\left(\Pi_{\taul} \uta(t_n+\xi)\right) ^2\|_{X^{s_{0}, 0}_{\tau}}
  \left\|
  \Pi_{\taul} \left[  \uta(t_n+s) +\psi_2(s,\partial_x)  \uta(t_n) \right] \right\|_{X^{0, {1 \over 2}}_{\tau}}.  $$
  Next, by using the discrete Strichartz estimate \eqref{discStrich}, we get that
  $$  \|\left(\Pi_{\taul} \uta(t_n+\xi)\right) ^2\|_{X^{s_{0}, 0}_{\tau}} \lesssim \| \uta(t_n+\xi)\|_{X^{s_{0}, {1 \over 2}}_{\tau}}^2.$$ 
  Consequently, by using again Lemma \ref{lemmadisc-cont}, we get that
  $$  \tau^{-1} \| \mathcal{E}^\tau_{3}(t_{n}) \|_{Y^0_\tau} \leq C_{T} \tau^{{s_{0} \over 3} + {1 \over 6}}.$$
If $s_{0} \geq 1$, we can directly use  \eqref{debutE2},  \eqref{embdisc1} and   \eqref{discStrich} to get that
$$  \tau^{-1} \| \mathcal{E}^\tau_{3}(t_{n}) \|_{Y^0_\tau}  \leq \tau^{1 \over 2} \sup_{s, \, \xi \in [0, \tau]} 
 \| \partial_{x}\uta(t_n+\xi)\|_{X^{0, {1 \over 2}}_{\tau}}  \| \uta(t_n+\xi)\|_{X^{0, {1 \over 2}}_{\tau}}
  (\| \uta(t_n+s)\|_{X^{0, {1 \over 2}}_{\tau}} +  \| \uta(t_n)\|_{X^{0, {1 \over 2}}_{\tau}})$$
  and therefore
  $$  \tau^{-1} \| \mathcal{E}^\tau_{3}(t_{n}) \|_{Y^0_\tau}  \leq  C_{T}\tau^{1 \over 2}.$$
   We thus have obtained that  
  $$  \tau^{-1} \| \mathcal{E}^\tau_{3}(t_{n}) \|_{Y^0_\tau}  \leq C_T \tau^{ {s_{0}\over 3}}$$
  if $s_{0} \leq 3/2.$
  
  If $s_{0}>3/2$, since $H^{s_{0}} \subset W^{1, \infty}$, we easily get directly from \eqref{defE3} that
  \begin{multline*}
   \tau^{-1} \| \mathcal{E}^\tau_{3}(t_{n}) \|_{Y^0_\tau} \leq \tau^{-1}\| \mathcal{E}^\tau_{3}(t_{n}) \|_{X^{0, 0}_{\tau}} \\
   \lesssim \tau  \sup_{\xi, \, s \in [0, \tau]} \| \Pi_{\tau} \partial_{xx} u(t_{n} + \xi ) \|_{l^2L^2} \|\partial_{x} u(t_{n} + \xi ) \|_{l^\infty  L^\infty}
   ( \|u(t_{n}+ s) \|_{l^\infty L^\infty} + \|u(t_{n})\|_{l^\infty L^\infty}).
   \end{multline*}
   This yields by using the property of $\Pi_{\tau}$ that for $s_{0} \leq 2$, we have
 $$ \tau^{-1} \| \mathcal{E}^\tau_{3}(t_{n}) \|_{Y^0_\tau} \leq  \tau^{ 1 - { 2 - s_{0} \over 3}} C_{T} \lesssim \tau^{ { 1 \over 3} + { s_{0} \over 3}}C_{T}$$
 while
 $$  \tau^{-1} \| \mathcal{E}^\tau_{3}(t_{n}) \|_{Y^0_\tau} \leq \tau C_{T}$$
 if $s_{0} \geq 2$.
 
 Summarizing all the cases, we obtain
   that
 $$  \tau^{-1} \| \mathcal{E}^\tau_{3}(t_{n}) \|_{Y^0_\tau}  \leq C_T \tau^{\alpha}.$$
 This ends the proof.
 \end{proof}

 \section{Proof of Theorem \ref{maintheo}}
 \label{sectionproofmain}
 
We are now in a position to give the proof of Theorem \ref{maintheo}.
We first observe that thanks to Proposition \ref{corNLSK}, we have from the triangle inequality that
\begin{equation}
\label{erreurtotale}
\|u(t_{n})- u^n\|_{L^2} \leq  \|u(t_{n}) - u_{\tau}(t_n)\|_{L^2} + \|u_{\tau}(t_{n}) - u^n\|_{L^2} \leq C_{T}\tau^{s_{0}   \over 3} + 
 \|u_{\tau}(t_{n}) - u^n\|_{L^2}.
\end{equation}
 To get the error estimates of Theorem \ref{maintheo} for $t_{n} \leq T_{1}$, it thus suffices to estimate $ \|e^n\|_{X^0_\tau}$ 
   thanks to \eqref{sob1/2} where $e^n$ solves the fixed point \eqref{enmod}.
  By using \eqref{bourg1/2}, we get that
 $$ \|e^n \|_{X^0_\tau}  \leq  \tau^{-1} \left\|   \mathcal{J}^\tau\left(  \eta\left({t_{\ell} \over T_{1}}\right)e^\ell, \eta\left({t_{\ell} \over T_{1}}\right)\uta(t_\ell) \right)\right\|_{Y^0_\tau}
  + \tau^{-1} \| \mathcal{E}(\tau, t_{n}) \|_{Y^0_\tau}$$
  and hence, thanks to Proposition \ref{propglobal} we have that
  $$  \|e^n \|_{X^0_\tau}  \leq  \tau^{-1}\left\|  \mathcal{J}^\tau\left(  \eta\left({t_{n }\over T_{1}}\right)e^n, \eta\left({t_{n} \over T_{1}}\right)\uta(t_n \right) \right\|_{Y^0_\tau}
  +  C_{T} \tau^{\alpha},$$
  where $\alpha = \min (1, s_{0}/3).$
 Next we estimate $ \tau^{-1}\left\|  \mathcal{J}^\tau\left(  \eta\left({t_{n} \over T_{1}}\right)e^n, \eta\left({t_{n} \over T_{1}}\right)\uta(t_n) \right)\right\|_{Y^0_\tau}$.
  From  the expression \eqref{J}, we get by using Lemma \ref{bourgaindurd} (and the fact that $\psi_{1}$, $\psi_{2}$
   are bounded Fourier multipliers) that
   \begin{multline*}
   \tau^{-1}\left\|  \mathcal{J}^\tau\left(  \eta\left({t_{n} \over T_{1}}\right)e^n, \eta\left({t_{n} \over T_{1}}\right)\uta(t_n )\right) \right\|_{Y^0_\tau}
   \lesssim   \left\| \eta\left({t_{n} \over T_{1}}\right)e^n \right\|_{X^{0, {1 \over 2}}_{\tau}}
    \left\| \eta\left({t_{n }\over T_{1}}\right) u_{\tau}(t_{n}) \right\|_{X^{0, {1 \over 3}}_{\tau}}
     \\ +   \left\| \eta\left({t_{n }\over T_{1}}\right)e^n \right\|_{X^{0, {1 \over 3}}_{\tau}}
    \left\| \eta\left({t_{n} \over T_{1}}\right) u_{\tau}(t_{n}) \right\|_{X^{0, {1 \over 2}}_{\tau}}
     +   \left\| \eta\left({t_{n} \over T_{1}}\right)e^n \right\|_{X^{0, {1 \over 2}}_{\tau}}
    \left\| \eta\left({t_{n} \over T_{1}}\right)  e^n \right\|_{X^{0, {1 \over 3}}_{\tau}}. 
   \end{multline*}
   Consequently, by using \eqref{bourg3} and Lemma \ref{lemmadisc-cont}, we get that
  $$  \tau^{-1}\left\|  \mathcal{J}^\tau\left(  \eta\left({t_{n} \over T_{1}}\right)e^n, \eta\left({t_{n} \over T_{1}}\right)\uta(t_n \right) \right\|_{Y^0_\tau}
   \leq C_{T} T_{1}^\epsilon \|e^n\|_{X^0_{\tau}} + C_{T} T_{1}^\epsilon \|e^n \|_{X^0_{\tau}}^2$$
   for some $C_{T}>0$ independent of $\tau \in (0, 1]$ and $T_{1}\in (0, 1]$.
   This yields 
   $$  \|e^n \|_{X^0_\tau}  \leq C_{T} \left(  \tau^{\alpha} + T_{1}^\epsilon \|e^n\|_{X^0_{\tau}}  + T_{1}^\epsilon \|e^n \|_{X^0_{\tau}}^2\right).$$
   We thus get for $T_{1}$ and $\tau$  sufficiently small that
   $$ \|e^n \|_{X^0_\tau}  \leq C_{T}  \tau^{\alpha}.$$
This proves the desired estimate \eqref{final1} for $ 0 \leq n \leq N_{1}= T_{1}/ \tau$. We can then iterate in a classical way the argument on  $ T_{1}/ \tau \leq  n \leq 2 T_{1}/\tau$ and so on to get the final estimate for $ 0 \leq n \leq T/\tau$.

\section{Proof of Lemma \ref{bourgaindurd}}
\label{sectiontechnical}
In this section, we shall prove Lemma \ref{bourgaindurd}. We adapt the proof in \cite{Bour93b}, \cite{KPV}, the main difficulty is to check that because of the frequency localization
 induced by the filter $\Pi_{\tau}$, the favorable frequency interaction of the KdV equation is kept at the discrete level.

The first step is to prove the following Strichartz estimate which has is own interest.
\begin{lemma}
\label{lemdiscStrich}
There exists $C>0$ such that  for every $u^{n} \in X^{0, {1 \over 3}}_{\tau}$, and $\tau \in (0, 1]$, we have
the estimate
$$ \left \| \Pi_{\tau} u^{n}\right\|_{l^4_{\tau}L^4}\leq C \|u^{n}\|_{X^{0, {1 \over 3}}_{\tau}}.$$
\end{lemma}
\begin{proof}
We first use a Littlewood-Paley type decomposition, 
we write
$$ \mathrm{1}_{[- {\pi \over \tau}, {\pi \over \tau})}(\sigma)= \sum_{m \geq 0} \mathrm{1}_{m}(\sigma),$$
where $ \mathrm{1}_{m}$ is supported in $2^{m} \leq  1 + | \sigma| <2^{m+1}\cap [- {\pi \over \tau}, {\pi \over \tau})$
(the sum is actually finite). Next, we extend $ \mathrm{1}_{m}$ on $\mathbb{R}$ by $2\pi/\tau$ periodicity
 so that
 $$ 1=  \sum_{m \geq 0} \mathrm{1}_{m}(\sigma).$$
 By using this decomposition, we expand
$$ \Pi_{\tau}u^n= \sum_{m  \geq 0} u^n_{m},$$
where 
$$\widetilde{u^n_{m}}(\sigma, k) = \widetilde{ u^n}(\sigma, k) \mathds{1}_{m}(\sigma, k)$$
and we have set
\begin{equation}
\label{def1m}  \mathds{1}_{m}(\sigma, k) = 
 \mathrm{1}_{m}(\sigma +  k^3) \mathrm{1}_{\mathcal{\tau}^{1 \over 3}|k|\leq 1} (k).
 \end{equation}
 Note that by our definition $\mathds{1}_{m}(\cdot, k)$ is $2\pi/\tau$ periodic.

   We then write
 $$ \| \Pi_{\tau} u^n \|_{l^4_{\tau}L^4}^2 = \| (\Pi_{\tau} u^n)^2 \|_{l^2_{\tau}L^2}
  \leq  2 \sum_{p \geq 0, \, q\geq 0 } \| u^n_{p}  u^n_{p+q}\|_{l^2_{\tau}L^2}$$
  and hence from the Bessel identity we have that
 \begin{equation}
 \label{sumfin}  \| \Pi_{\tau} u^n \|_{l^4_{\tau}L^4}^2  \leq  2 \sum_{p \geq 0, \, q\geq 0 }
  \| \widetilde{u^n_{p}} * \widetilde{u^n_{p+q}} (\sigma, k) \|_{L^2l^2},
  \end{equation}
  where 
  $$\widetilde{u^n_{p}} * \widetilde{u^n_{p+q}} (\sigma, k)=
   \sum_{k'} \int_{\sigma'} \widetilde{u^n_{p}} (\sigma', k')  \widetilde{u^n_{p+q}}(\sigma - \sigma', k-k') d\sigma'.$$
We thus need to estimate $ \| \widetilde{u^n_{p}} * \widetilde{u^n_{p+q}} (\sigma, k) \|_{L^2l^2}$.
We shall handle differently  the $l^2$ norm for $|k| \leq 2^\beta$ and $|k | \geq 2^\beta$ for $\beta$ to be chosen.

 For $|k| \leq 2^\beta$, we write
 $$  \| \widetilde{u^n_{p}} * \widetilde{u^n_{p+q}} (\sigma, k) \|_{L^2(\sigma)}
  \leq  \sum_{k'}  \left\| \int_{\sigma'}\widetilde{u^n_{p}} (\sigma', k')  \widetilde{u^n_{p+q}}(\sigma - \sigma', k-k') d \sigma' \right \|_{L^2(\sigma)} $$
  and we use the Young inequality for convolution to obtain
  $$   \| \widetilde{u^n_{p}} * \widetilde{u^n_{p+q}} (\sigma, k) \|_{L^2(\sigma)} \leq
  \sum_{k'}   \int_{\sigma'} \left\| \widetilde{u^n_{p}}(\cdot, k')\right\|_{L^1(\sigma)} \left\| \widetilde{u^n_{p+q}}(\cdot, k-k') \right\|_{L^2(\sigma)}.$$
  By the frequency localization of  $\widetilde{u^n_{p}}$, we have by Cauchy-Schwarz that
  $$ \left\| \widetilde{u^n_{p}}(\cdot, k')\right\|_{L^1(\sigma)}  \lesssim  2^{p \over 2} \left\| \widetilde{u^n_{p}}(\sigma, k')\right\|_{L^2(\sigma)}.$$ 
  Therefore we obtain by using also Cauchy-Schwarz for the sum in $k_{1}$ that
  $$  \| \widetilde{u^n_{p}} * \widetilde{u^n_{p+q}} (\sigma, k) \|_{L^2(\sigma)} \lesssim 
  2^{p \over 2}   \left\| \widetilde{u^n_{p}}\right\|_{L^2(\sigma)l^2}  \left\| \widetilde{u^n_{p+q}} \right\|_{L^2(\sigma)l^2}.$$
  This yields
\begin{equation}
\label{kpetit*}
 \| \widetilde{u^n_{p}} * \widetilde{u^n_{p+q}} \|_{L^2(\sigma)l^2(|k| \leq 2^\beta)}
  \lesssim 2^{ \beta + p \over 2}  \left\| \widetilde{u^n_{p}}\right\|_{L^2(\sigma)l^2}  \left\| \widetilde{u^n_{p+q}} \right\|_{L^2(\sigma)l^2}.
\end{equation}
For $|k| \geq 2^\beta$, we write that 
$$ \| \widetilde{u^n_{p}} * \widetilde{u^n_{p+q}} \|_{L^2(\sigma)l^2(|k| \geq 2^\beta)}
=\left\| \sum_{k'} \int_{\sigma'} \widetilde{u^n_{p}} (\sigma', k')  \widetilde{u^n_{p+q}}(\sigma - \sigma', k-k') 
 \mathds{1}_{p}(\sigma', k') \mathds{1}_{p+q}(\sigma - \sigma', k-k') d \sigma' \right \|_{L^2(\sigma) l^2(|k| \geq 2^\beta)}
$$
and we get from Cauchy-Schwarz that
$$  \| \widetilde{u^n_{p}} * \widetilde{u^n_{p+q}} \|_{L^2(\sigma)l^2(|k| \geq 2^\beta)} \leq 
 \left\|\left( \sum_{k'} \int_{\sigma'} |\widetilde{u^n_{p}} (\sigma', k')|^2 |  \widetilde{u^n_{p+q}}(\sigma - \sigma', k-k') |^2d \sigma'
 \right)^{ 1 \over 2} \left( \mathds{1}_{p}* 1_{p+q}(\sigma, k)\right)^{1 \over 2} \right\|_{L^2(\sigma) l^2(|k| \geq 2^\beta)}.$$
 This yields
 $$  \| \widetilde{u^n_{p}} * \widetilde{u^n_{p+q}} \|_{L^2(\sigma)l^2(|k| \geq 2^\beta)} \leq 
 \left\| (\mathds{1}_{p}* \mathds{1}_{p+q})^{1 \over 2}\right\|_{L^\infty( \sigma, |k| \geq 2^\beta)}
  \left\| \widetilde{u^n_{p}}\right\|_{L^2(\sigma)l^2}  \left\| \widetilde{u^n_{p+q}} \right\|_{L^2(\sigma)l^2}.$$
   We thus need to estimate  $\left\| (\mathds{1}_{p}* \mathds{1}_{p+q})^{1 \over 2}\right\|_{L^\infty( \sigma, |k| \geq 2^\beta)}$
   where 
   $$\mathds{1}_{p}* \mathds{1}_{p+q}(\sigma, k)
    = \sum_{k'}\int_{\sigma' \in [-\pi/\tau, \pi/\tau]} \mathds{1}_{p}(\sigma', k') \mathds{1}_{p+q}(\sigma - \sigma', k-k') d \sigma' .$$
From the definitions of $\mathds{1}_{m}$, we have a non-zero integral
 if $ |\sigma'  + k'^3  - 2{m_{1} \pi \over \tau} | \leq 2^{p+1}$ and $  | \sigma - \sigma'  + (k- k)'^3  - 2{m_{2} \pi \over \tau} | \leq 2^{p+ q+1}$
 for some $m_{1}$, $m_{2 } \in \mathbb{Z}$
  which means that 
 $\sigma' +  k'^3 \in  E_{p}$ and  $ \sigma - \sigma'  + (k- k)'^3 \in E_{p+q}$
 where we have set  $E_{l} =\cup_{|m| \leq N \in \mathbb{Z}}\left[ - 2^{l+1} + {2 m \pi \over \tau},  2^{l+1} +{2 m \pi \over \tau}\right] $. 
 Note that since $|k|, \, |k'| \leq \tau^{-{1 \over 3}}$ the number of intervals in $E_{p}$  and $E_{p+q}$ yielding
 a nontrivial contribution is $\mathcal{O}(1)$, we can thus  take $N= \mathcal{O}(1)$  independent of $\tau$.
For a given $k'$, we observe that  if the integral is not zero then it is bounded  by $ \mathcal{O}(2^{p})$.
 Moreover, to evaluate the number of non zero terms in the sum, we see that, we must have 
 $$ \sigma + k'^3  + (k-k')^3 \in E_{p+q} + E_{p}$$ which is equivalent to 
 $$  k'^2 -  k k' \in { 1 \over 3}\left( - { \sigma \over   k} - k^2 +  2^{-\beta} ( E_{p+q} + E_{p})\right)$$
 since $|k| \geq 2^\beta$. This yields
 $$   \left(k' - {k \over 2} \right)^2 \in  - {k^2 \over 4}  +  { 1 \over 3}\left( - { \sigma \over   k} - k^2 +  2^{-\beta} ( E_{p+q} + E_{p})\right)$$
 which means that $k'$ must be restricted to a finite number $N$ of intervals of length smaller than
  $ \mathcal{O}(2^{ p+q  -  \beta \over 2})$.  We thus find that
  $$  \left\| (\mathds{1}_{p}* \mathds{1}_{p+q})^{1 \over 2}\right\|_{L^\infty( \sigma, |k| \geq 2^\beta)}
   \lesssim 2^{p \over 2} 2^{ p+q  -  \beta \over 4} $$
   and hence 
  $$ \| \widetilde{u^n_{p}} * \widetilde{u^n_{p+q}} \|_{L^2(\sigma)l^2(|k| \geq 2^\beta)} \lesssim   2^{  3p+q - \beta \over 4}
   \left\| \widetilde{u^n_{p}}\right\|_{L^2(\sigma)l^2}  \left\| \widetilde{u^n_{p+q}} \right\|_{L^2(\sigma)l^2}.$$
Thanks to the last estimate and \eqref{kpetit*}, we can then optimize the choice of $\beta$.
 We take $\beta = {p + q \over 3}$ and we deduce that
 $$   \| \widetilde{u^n_{p}} * \widetilde{u^n_{p+q}} \|_{L^2(\sigma)l^2(k)} \lesssim 
 2^{ 4p + q \over 6} \left\| \widetilde{u^n_{p}}\right\|_{L^2(\sigma)l^2}  \left\| \widetilde{u^n_{p+q}} \right\|_{L^2(\sigma)l^2}.
 $$
 We then get from \eqref{sumfin} that
 $$ \| \Pi_{\tau} u^n \|_{l^4_{\tau}L^4}^2  \lesssim 
  \sum_{q \geq 0}  2^{-q \over 6} \sum_{p \geq 0} 2^{p \over 3}   \left\| \widetilde{u^n_{p}}\right\|_{L^2(\sigma)l^2}
   2^{(p+q) \over 3}  \left\| \widetilde{u^n_{p+q}} \right\|_{L^2(\sigma)l^2}.$$
    We finally conclude by using Cauchy-Schwarz for the sum in $q$ and the fact that 
    $$ \sum_{l} \left( 2^{l \over 3}  \left\| \widetilde{u^n_{l}}\right\|_{L^2(\sigma)l^2}\right)^2
     \lesssim \| u^n \|_{X^{0, {1 \over 3}}_{\tau}}^2.$$
    
\end{proof}

We can then deduce from Lemma \ref{lemdiscStrich} that
\begin{cor}
For every $s \geq 0$, 
there exists $C>0$ such that  for every $u^{n}, \, v^n \in X^{s, {1 \over 3}}_{\tau}$, and $\tau \in (0, 1]$, we have
the estimate
\begin{equation}
\label{discStrich} \left \| \langle \partial_{x} \rangle^s (\Pi_{\tau} u^{n} \Pi_{\tau} v^n)\right\|_{l^2_{\tau}L^2}\leq C \|u^{n}\|_{X^{s, {1 \over 3}}_{\tau}}.
\end{equation}
\end{cor}

\begin{proof}
We observe that:
\begin{multline*} 
 \left \| \langle \partial_{x} \rangle^s (\Pi_{\tau} u^{n} \Pi_{\tau} v^n)\right\|_{l^2_{\tau}L^2}
 = \left \| \langle k  \rangle^s ( \widetilde{\Pi_{\tau} u^{n}}  *\widetilde{ \Pi_{\tau} v^n})(\sigma, k)\right\|_{L^2l^2} \\
 \lesssim   \left \|  \left( \langle k\rangle^s | \widetilde{\Pi_{\tau} u^{n}} | \right) * | \widetilde{ \Pi_{\tau} v^n}|\right\|_{L^2l^2} 
  +   \left \|  | \widetilde{\Pi_{\tau} u^{n}} |  *  \left( \langle k\rangle^s| \widetilde{ \Pi_{\tau} v^n} |\right)\right\|_{L^2l^2}.
 \end{multline*} 
To conclude, we observe that
$$   \left \|  \left( \langle k\rangle^s | \widetilde{\Pi_{\tau} u^{n}} | \right) * | \widetilde{ \Pi_{\tau} v^n}|\right\|_{L^2l^2} 
= \| a^n b^n \|_{l^2_{\tau}L^2},$$
where $a^n(x)$, $b^n(x)$ are such that
$$ \widetilde{a^n} (\sigma, k)=  \langle k\rangle^s | \widetilde{\Pi_{\tau} u^{n}} |, \quad
  \widetilde{b^n} (\sigma, k)= | \widetilde{\Pi_{\tau} v^{n}} |.$$
  By using Cauchy-Schwarz and Lemma \ref{lemdiscStrich}, we get that
  $$ \left \|  \left( \langle k\rangle^s | \widetilde{\Pi_{\tau} u^{n}} | \right) * | \widetilde{ \Pi_{\tau} v^n}|\right\|_{L^2l^2} 
   \leq \|\Pi_{\tau} a^n \|_{l^4_{\tau}L^4} \|\Pi_{\tau} n^n \|_{l^4_{\tau}L^4}
    \lesssim \|a^n \|_{X^{0, {1 \over 3}}_{\tau}} \|b^n \|_{X^{0, {1 \over 3}}_{\tau}}
     \lesssim  \|u^n \|_{X^{s, {1 \over 3}}_{\tau}} \|v^n \|_{X^{0, {1 \over 3}}_{\tau}}.$$
The symmetric term can be handled with a similar argument. This ends the proof.

\end{proof}

We are now in position to give the proof of Lemma \ref{bourgaindurd}.
\subsection*{Proof of Lemma \ref{bourgaindurd}}
We give the proof for $s=0$. The general case $s>0$ can be easily deduced with the same type of arguments
as above. We shall thus estimate
 $\| \partial_{x}\Pi_{\tau}( \Pi_{\tau}u^n  \Pi_{\tau}v^n) \|_{Y^0_{\tau}}$ for $u^n, \, v^n \in X^0_{\tau}.$ We recall that our definition of this space contain the fact
 that the functions have zero mean for every time.
 
 We start with estimating $ \|\partial_{x} \Pi_{\tau}( \Pi_{\tau}u^n \Pi_{\tau} v^n) \|_{X^{0, {- {1 \over 2}}}_{\tau}}.$
 We use a duality argument:
 \begin{multline*}  \|\partial_{x}\Pi_{\tau}(\Pi_{\tau} u^n \Pi_{\tau} v^n ) \|_{X^{0, {- {1 \over 2}}}_{\tau}} = \sup_{ \|w^n \|_{X^{0, {1 \over 2}}_{\tau}}\leq 1} \left|
  \tau \sum_{n} \int_{\mathbb{T}} \Pi_{\tau}u^n  \Pi_{\tau}v^n\partial_{x} \Pi_{\tau}w^n\, dx \right|
  \\= \sup_{ \|w^n \|_{X^{0, {1 \over 2}}_{\tau}}\leq 1} \left|
  \sum_{k, \, k'} \int_{\sigma, \, \sigma '}  k\, 
  \widetilde{\Pi_{\tau}u^n}(\sigma', k')  \widetilde{\Pi_{\tau}v^n}(\sigma- \sigma', k - k')  \widetilde{\Pi_{\tau}w^n}( - \sigma,  - k) \, 
  d\sigma' d \sigma
   \right|.
   \end{multline*}
   We then set
  \begin{align*}
  &  \widetilde{a^n}(\sigma, k)= \langle d_{\tau}(\sigma + k^3) \rangle^{1 \over 2} |  \widetilde{\Pi_{\tau}u^n}(\sigma, k)|, \\
    &    \widetilde{b^n}(\sigma, k)= \langle d_{\tau}(\sigma + k^3)\rangle^{1 \over 2} |  \widetilde{\Pi_{\tau}v^n}(\sigma, k)|, \\
  &    \widetilde{c^n}(\sigma, k)=   \langle d_{\tau}(\sigma + k^3)\rangle^{1 \over 2} |  \widetilde{\Pi_{\tau}w^n}(-\sigma, -k)|,
    \end{align*}
   and we shall estimate 
 $$ I = 
  \sum_{k, \, k'} \int_{\sigma\, \sigma'}   m_{\tau}(\sigma, \sigma', k, k')     \widetilde{a^n}(\sigma', k') \widetilde{b^n}(\sigma- \sigma', k-k') \widetilde{c^n} (\sigma, k)\, d \sigma d \sigma',
    $$
    where 
  $$ m_{\tau}(\sigma, \sigma', k, k')= { |k| 
    \over 
    \langle d_{\tau}(\sigma' + k'^3) \rangle^{1 \over 2} \langle d_{\tau}(\sigma- \sigma' + (k- k')^3)) \rangle^{1 \over 2} \langle d_{\tau}(\sigma+ k^3) \rangle^{1 \over 2}}.$$
Note that we have $\sigma, \, \sigma ' \in [- \pi/\tau, \pi/\tau]$ and by the choice of $\Pi_{\tau}$,  $|k|^3, \, |k'|^3 \leq \tau^{-1}.$

We first assume  that  
$ |\sigma + k^3|$ or $ | \sigma ' + k'^3|$  is bigger than $\epsilon/\tau$ for some $\epsilon>0$ independent
 of $\tau$ to be chosen. Let us  assume that it is the first
one (the other case being symmetric), then $ \langle d_{\tau} (\sigma + k^3) \rangle^{ 1 \over 2} \gtrsim \tau^{- { 1 \over 2}}$
and therefore,
\begin{multline}
\label{tresbon}  m_{\tau}(\sigma, \sigma', k, k') \lesssim { \tau^{ 1 \over 6}
\over 
    \langle d_{\tau}(\sigma' + k'^3) \rangle^{1 \over 2} \langle d_{\tau}(\sigma- \sigma' + (k- k')^3)) \rangle^{1 \over 2}}
\\ \lesssim   { 1
\over 
    \langle d_{\tau}(\sigma' + k'^3) \rangle^{1 \over 2} \langle d_{\tau}(\sigma- \sigma' + (k- k')^3)) \rangle^{1 \over 2}}.
    \end{multline}
When both $ |\sigma + k^3|$ and $| \sigma ' + k'^3|$ are smaller than $\epsilon/\tau$, since
$$ \sigma- \sigma' + (k- k')^3 = \sigma + k^3 - (\sigma' + k^3) -3kk'(k-k'), $$
we have that
$$  \left |\sigma- \sigma' + (k- k')^3 \right| \leq 2 {\epsilon \over \tau} + {6 \over \tau} < { 2 \pi \over \tau}$$
by choosing $\epsilon$ sufficiently small. Therefore in this situation we have that
$$ |d_{\tau} (\sigma + k^3) \gtrsim  | \sigma + k^3 |, \,  |d_{\tau} (\sigma' + k'^3) \gtrsim   | \sigma' + k^3 |, 
\, | d_{\tau}(\sigma- \sigma' + (k- k')^3)) | \gtrsim | \sigma- \sigma' + (k- k')^3|.$$
We are thus in a situation very close to the continuous case, we have
$$  m_{\tau}(\sigma, \sigma', k, k') \lesssim
{ |k| 
    \over 
    \langle\sigma' + k'^3  \rangle^{1 \over 2} \langle \sigma- \sigma' + (k- k')^3 \rangle^{1 \over 2} \langle \sigma+ k^3 \rangle^{1 \over 2}}$$
 and since
 $$  \sigma' + k'^3 - (\sigma + k^3) + \sigma- \sigma' + (k- k')^3 = - 3 kk'(k-k')$$
 we deduce that
  \begin{equation}
  \label{intercont}\max( |\sigma' + k'^3 |,  |\sigma + k^3|, |\sigma- \sigma' + (k- k')^3 |)   \geq 3 |k|\, |k'|\, |k-k'|.
  \end{equation}
   Let us  assume that the largest one above  is $ |\sigma+ k^3|$, the other cases being similar. Then we get
   \begin{multline*}
    m_{\tau}(\sigma, \sigma', k, k') \lesssim { |k|^{1 \over 2} \over  \langle k'\rangle^{ 1 \over 2} \langle k-k' \rangle^{1 \over 2}
    \langle d_{\tau}(\sigma' + k'^3) \rangle^{1 \over 2} \langle d_{\tau}(\sigma- \sigma' + (k- k')^3) \rangle^{1 \over 2} } 
    \\ \lesssim { 1 \over 
    \langle d_{\tau}(\sigma' + k'^3) \rangle^{1 \over 2} \langle d_{\tau}(\sigma- \sigma' + (k- k')^3) \rangle^{1 \over 2}} 
    \end{multline*}
    by using that $|k| \leq |k'| + |k-k'|$ to get the last line.
     The above estimate is similar to \eqref{tresbon}.
      We can thus estimate $I$ by $II$+ symmetric terms where
    $$ II =  
  \sum_{k, \, k'} \int_{\sigma, \, \sigma'}      \widetilde{\alpha^n}(\sigma', k') \widetilde{\beta^n}(\sigma- \sigma', k-k') \widetilde{c^n} (\sigma, k)\, d \sigma d \sigma',$$
  where we have set
 $$  \widetilde{\alpha^n}(\sigma, k)  =
  |\widetilde{\Pi_{\tau}u^n}(\sigma, k)|, \quad
        \widetilde{\beta^n}(\sigma, k)=  |  \widetilde{\Pi_{\tau}v^n}(\sigma, k)|.
 $$
 Going back to the physical space, we get that
 $$ II = \tau \sum_{n} \int_{\mathbb{T}} \alpha^n\beta^n  c^n \, dx$$
 and hence from the H\"older inequality, we find
 $$ II \leq  \|\alpha^n \|_{l^4_{\tau}L^4} \|\beta^n \|_{l^4_{\tau}L^4} \| c^n \|_{l^2_{\tau}L^2}.$$
 This yields thanks to Lemma \ref{lemdiscStrich}, 
 $$ II \lesssim \| \alpha ^n \|_{X^{0, {1 \over 3}}_{\tau}} \| \beta ^n \|_{X^{0, {1 \over 3}}_{\tau}}
  \| c ^n \|_{l^2_{\tau}L^2} \lesssim  \| u^n \|_{X^{0, {1 \over 3}}_{\tau}} \| v^n \|_{X^{0, {1 \over 3}}_{\tau}} \|w^n\|_{X^{0, {1 \over 2}}_{\tau}}.$$
  We thus finally get that
  $$ I \leq  (\| u^n \|_{X^{0, {1 \over 3}}_{\tau}} \| v^n \|_{X^{0, {1 \over 2}}_{\tau}}  + \| u^n \|_{X^{0, {1 \over 2}}_{\tau}} \| v^n \|_{X^{0, {1 \over 2}}_{\tau}})\|w^n\|_{X^{0, {1 \over 2}}_{\tau}}$$
  from which we deduce the estimate of $ \|\partial_{x}\Pi_{\tau}(\Pi_{\tau} u^n \Pi_{\tau} v^n ) \|_{X^{0, {- {1 \over 2}}}_{\tau}}$.
  
  It remains to estimate  $\left\|{ 1 \over d_{\tau}(\sigma + k^3)} \mathcal{F}_{n,x \rightarrow \sigma, k}\left(\partial_{x}\Pi_{\tau}(\Pi_{\tau} u^n \Pi_{\tau} v^n ) \right) \right\|_{l^2(k)L^1(\sigma)}.$
   We use again a duality argument, we take $(w^n)_{n}$ such
   that $\| \widetilde{w^n} \|_{l^2(k) L^\infty(\sigma)} \leq 1$ and
   hence, we have to estimate with $a^n$ and $b^n$ as above
   $$ III= \sum_{k, \, k'} \int_{\sigma\, \sigma'}   m^1_{\tau}(\sigma, \sigma', k, k')     \widetilde{a^n}(\sigma', k') \widetilde{b^n}(\sigma- \sigma', k-k')  |\widetilde{\Pi_{\tau}w^n}| (\sigma, k)\, d \sigma d \sigma'$$
 with
 $$  m^1_{\tau}(\sigma, \sigma', k, k') =   
   { |k|   \over 
    \langle d_{\tau}(\sigma' + k'^3) \rangle^{1 \over 2} \langle d_{\tau}(\sigma- \sigma' + (k- k')^3)) \rangle^{1 \over 2} \langle d_{\tau}(\sigma+ k^3) \rangle}.$$
    Again, if 
$ |\sigma' + k'^3|\geq \epsilon/\tau$, 
we have since $|k| \leq \tau^{- { 1 \over 3}}$ that
\begin{equation}
\label{kpetit}  m^1_{\tau}(\sigma, \sigma', k, k')  \lesssim   
   { 1 \over 
    \langle d_{\tau}(\sigma- \sigma' + (k- k')^3)) \rangle^{1 \over 2} \langle d_{\tau}(\sigma+ k^3) \rangle}.
    \end{equation}
  Going back to the physical space  and using the H\"older inequality, we estimate this part of  $III$
     by 
  $$ \|a^n \|_{l^2_{\tau}L^2} \|\beta^n \|_{l^4_{\tau}L^4}  \left \| \mathcal{F}_{\sigma,k\rightarrow n, x}\left({ |\widetilde{\Pi_{\tau}w^n}|
  \over  \langle d_{\tau}(\sigma+ k^3) \rangle }\right) \right\|_{l^4_{\tau}L^4} $$
    which is bounded thanks to Lemma \ref{lemdiscStrich} by
  \begin{multline*}  \|a^n \|_{l^2_{\tau}L^2}  \|\beta^n \|_{X^{0, {1 \over 3}}_{\tau}}  \left\|  {\widetilde{w^n}
  \over  \langle d_{\tau}(\sigma+ k^3) \rangle^{2 \over 3} }\right\|_{l^2(k)L^2(\sigma)}
  \\ \lesssim \|u^n\|_{X^{0, {1 \over 2}}_{\tau}}\|v^n\|_{X^{0, {1 \over 3}}_{\tau}} 
 \left\| {\widetilde{w^n}
  \over  \langle d_{\tau}(\sigma+ k^3) \rangle^{2 \over 3}} \right\|_{l^2(k)L^2(\sigma)}
   \lesssim\|u^n\|_{X^{0, {1 \over 2}}_{\tau}}\|v^n\|_{X^{0, {1 \over 3}}_{\tau}} 
 \left\|\widetilde{w^n}\right\|_{l^2(k)l^\infty(\sigma)} 
 \end{multline*}
 since $4/3>1$ which is the desired estimate.
 
  If   $ | \sigma  + k^3|\geq \epsilon/\tau$, we have
  $$m^1_{\tau}(\sigma, \sigma', k, k')  \lesssim_{\epsilon}   
   { \tau^{- { 1 \over 3}}  \over 
  \langle d_{\tau}(\sigma' + k'^3) \rangle^{1 \over 2}    \langle d_{\tau}(\sigma- \sigma' + (k- k')^3) \rangle^{1 \over 2} \langle |d_{\tau}(\sigma+ k^3)| + { 1 \over \tau } \rangle } .$$ 
  From the same arguments as above  using Lemma \ref{lemdiscStrich}, we then get that this contribution in $III$
   can be estimated by 
   $$ \tau^{- {1 \over 3} } \|u^n\|_{X^{0, {1 \over 3}}_{\tau}}\|v^n\|_{X^{0, {1 \over 3}}_{\tau}} 
 \left\| {\widetilde{w^n}
  \over  \langle |d_{\tau}(\sigma+ k^3)| + { 1 \over \tau} \rangle }\right\|_{l^2(k)L^2(\sigma)}
   \lesssim  \|u^n\|_{X^{0, {1 \over 3}}_{\tau}}\|v^n\|_{X^{0, {1 \over 3}}_{\tau}} 
 \left\| \widetilde{w^n} \right\|_{l^2(k) L^\infty(\sigma)},$$
 where for the last estimate, we have used that
 $$ \int_{- {\pi\over \tau}}^{\pi \over \tau} { 1 \over  \langle |d_{\tau}(\sigma+ k^3)| + { 1 \over \tau} \rangle^{2}} 
  \lesssim  \int_{- {\pi\over \tau}}^{\pi \over \tau} { 1 \over  1+  |\sigma|^{2} +   { 1 \over \tau^{2}  } }\, d\sigma
   \lesssim \tau.$$
   It remains to handle the case  $ | \sigma  + k^3|\leq  \epsilon/\tau$, $ |\sigma' + k'^3|\leq \epsilon/\tau$.
    By choosing $\epsilon$ sufficiently small as before, we have in this case that
    $$ m^1_{\tau}(\sigma, \sigma', k, k')  \lesssim { |k|   \over 
    \langle \sigma' + k'^3 \rangle^{1 \over 2} \langle \sigma- \sigma' + (k- k')^3 \rangle^{1 \over 2} \langle \sigma+ k^3 \rangle}.$$
    We shall thus use again the property \eqref{intercont}.
     We shall consider the two cases:
     \begin{itemize}
     \item if $\max( |\sigma' + k'^3 |,  |\sigma + k^3|, |\sigma- \sigma' + (k- k')^3 |) =  | \sigma' + k'^3 |$
      (the case that  the max is $|\sigma- \sigma' + (k- k')^3 |$ is  symmetric). Then we get from \eqref{intercont}   that
      $$  m^1_{\tau}(\sigma, \sigma', k, k')  \lesssim { 1 \over 
    \langle d_{\tau}(\sigma- \sigma' + (k- k')^3)
     \rangle^{1 \over 2} \langle d_{\tau}(\sigma+ k^3) \rangle} $$
     which is similar to \eqref{kpetit}. We can thus estimate this contribution to $III$ in the same way as previously.
     \item if $ \max( |\sigma' + k'^3 |,  |\sigma + k^3|, |\sigma- \sigma' + (k- k')^3 |) =  |\sigma+ k^3 |$.
     We observe that \eqref{intercont} gives
     $$  |  \sigma+ k^3 | \gtrsim |k| |k'| |k-k'| \gtrsim  |k|^2$$
     and we therefore get that
     $$  m^1_{\tau}(\sigma, \sigma', k, k')  \lesssim { |k| \over  \langle d_{\tau}(\sigma' +  k'^3) \rangle^{1 \over 2} 
    \langle d_{\tau}(\sigma- \sigma' + (k- k')^3) \rangle^{1 \over 2} \langle |d_{\tau}(\sigma+ k^3)| + |k|^2 \rangle}.$$
    We thus get by using again Lemma \ref{lemdiscStrich},  get that this contribution in $III$
   can be estimated by 
   $$  \|u^n\|_{X^{0, {1 \over 3}}_{\tau} }\|v^n\|_{X^{0, {1 \over 3} }_{\tau} } 
 \left\| { |k| \,\widetilde{w^n} \over  \langle |d_{\tau}(\sigma+ k^3)| + |k|^2\rangle }\right\|_{l^2(k)L^2(\sigma)}.$$
  To conclude, we use that 
  $$   \left\| { |k| \,\widetilde{w^n}   \over  \langle |d_{\tau}(\sigma+ k^3)| + |k|^2 \rangle }  \right\|_{l^2(k)L^2(\sigma) }
  \lesssim  \left\| \widetilde{w^n} \right\|_{l^2(k) L^\infty(\sigma)}$$
  since 
 $$ \int_{- {\pi\over \tau}}^{\pi \over \tau} { 1 \over  \langle |d_{\tau}(\sigma+ k^3)| + |k|^2 \rangle^{2} } d\sigma
  \lesssim  \int_{- {\pi\over \tau}}^{\pi \over \tau} { 1 \over  1+  |\sigma|^2 +   |k|^4 }\, d\sigma
   \lesssim  { 1 \over |k|^2}.$$
     
     \end{itemize}
   Gathering all  the above estimates, we arrive at   
 $$  \left\|{ 1 \over d_{\tau}(\sigma + k^3)} \mathcal{F}_{n,x \rightarrow \sigma, k}\left(\partial_{x}\Pi_{\tau}(\Pi_{\tau} u^n \Pi_{\tau} v^n ) \right) \right\|_{l^2(k)L^1(\sigma)}
  \lesssim  \|u^n\|_{X^{0, {1 \over 2}}_{\tau} }\|v^n\|_{X^{0, {1 \over 3} }_{\tau} } +  \|u^n\|_{X^{0, {1 \over 3}}_{\tau} }\|v^n\|_{X^{0, {1 \over 2} }_{\tau} } .$$
  This ends the proof of Lemma \ref{bourgaindurd}.

\subsection*{Acknowledgements}

{\small
KS has received funding from the European Research Council (ERC) under the European Union’s Horizon 2020 research and innovation programme (grant agreement No. 850941).
}

\end{document}